\documentclass[11pt,reqno]{amsart}

\usepackage[margin = 1.0in]{geometry}

\usepackage{fourier}
\usepackage{amssymb}
\usepackage[hidelinks]{hyperref}

\numberwithin{equation}{section}

\newtheorem{theorem}{Theorem}[section]
\newtheorem{lemma}[theorem]{Lemma}
\newtheorem{proposition}[theorem]{Proposition}
\newtheorem{corollary}[theorem]{Corollary}
\theoremstyle{definition}

\theoremstyle{remark}
\newtheorem{remark}[theorem]{Remark}

\newcommand{\R}{\mathbb R}

\newcommand{\F}{\mathrm F}
\newcommand{\HS}{\mathrm{HS}}
\newcommand{\codim}{\operatorname{codim}}
\newcommand{\conv}{\operatorname{conv}}
\newcommand{\diag}{\operatorname{diag}}
\newcommand{\rank}{\operatorname{rank}}
\newcommand{\spanop}{\operatorname{span}}
\newcommand{\supp}{\operatorname{supp}}
\newcommand{\tr}{\operatorname{tr}}

\title[Variational sparsification and geometric applications]{Variational sparsification, functional Helly theorems, and distance to the cube}

\author{Omer Friedland}

\address{Institut de Math\'ematiques de Jussieu, Sorbonne Universit\'e,
Campus Pierre et Marie Curie, 4 Place Jussieu,
75252 Paris cedex 05, France}

\email{omer.friedland@imj-prg.fr}

\subjclass[2020]{Primary 52A23; Secondary 46B20, 52A35, 15A42, 60B20}

\keywords{Spectral sparsification, exact linear constraints, covariance matrices,
functional John ellipsoids, quantitative Helly theorems, Banach-Mazur distance,
Gaussian polytopes}

\date{}

\hypersetup{
 pdftitle = {Variational sparsification, functional Helly theorems, and distance to the cube},
 pdfauthor = {Omer Friedland},
 pdfsubject = {Variational sparsification and geometric applications}
}

\begin{document}

\begin{abstract}
We develop variational methods for constructing sparse representations
that retain exact equations or inverse-moment estimates needed in
convex geometry. First variations of penalized spectral objectives
give quantitative estimates, while second variations bound the
dimension of feasible coefficient changes. For positive semidefinite
sums of rank $r$, the method preserves prescribed linear equations
with relative spectral error $\varepsilon$ and nullity at most
$Cr/\varepsilon^2$ for the constraint matrix restricted to the final
support. The construction extends to proper cones equipped with
self-concordant barriers.

For bases in a prescribed symmetric convex hull, linear restrictions
produce nilpotent variations, making the inverse basis matrix exactly
affine and yielding inverse moments and coefficient bounds at the
same minimum. A general covariance family explains the different
curvature at the inverse-trace and determinant endpoints.

The fixed-summand construction gives a functional Helly theorem:
from a finite log-concave family whose minimum has finite positive
integral, one can select $O(n)$ original functions with integral
loss at most $(C\sqrt n)^n$, of optimal dimensional order for a
fixed linear selection count. The inverse-trace construction, combined with a
uniform Gaussian perturbation theorem, gives distance at least
$cn^{2/3}/(\log n)^{5/6}$ from the cube with exponentially high
probability for polars of symmetric Gaussian polytopes with $2n$
independent generators. Together with the upper bound of
Friedland-Naor-Youssef, whose determinant optimization is the other
spectral endpoint, this determines the polynomial exponent $2/3$
of the Banach-Mazur radius of the cube.
\end{abstract}

\maketitle

\setcounter{tocdepth}{1}
\tableofcontents

\section{Introduction}

A sparse representation in convex geometry may have to retain more
than a covariance matrix. Contact equations make an extremal ellipsoid
useful for comparison with every competing ellipsoid. Directional
inverse moments can locate the part of a representation that a
probabilistic argument must control. Relative spectral approximation
alone does not retain all this information: even a small relative
change in a covariance can produce an arbitrarily large change in a
higher inverse moment in a fixed direction; see
Remark~\ref{framework:inverse-moment-example}.

Our approach is to obtain representations and their quantitative
estimates from the same variational construction. We couple a spectral objective
with a concave penalty on the coefficients. Exact equations are built
into the feasible set. First variations supply geometric estimates and
bounds on spectral curvature. Second variations on suitable linear
spaces then bound the number of feasible coefficient directions at the
minimum. The choices that make this effective are a metric adapted to
the exact equations for fixed summands, and variations that simplify
the inverse exactly when the vectors themselves move.

The latter construction can be seen directly. Write a variation of an
inscribed basis as $X(s) = (I+sQ)X$. If the changing coefficients use
generators in a subspace $E$, then $\operatorname{range}Q\subseteq E$.
Imposing $Q|_E = 0$ costs at most $(\dim E)^2$ linear equations and gives
$$
 Q^2 = 0,\quad (I+sQ)^{-1} = I-sQ.
$$
The inverse basis matrix $X(s)^{-1} = X^{-1}(I-sQ)$ is therefore exactly
affine in $s$. On this space, the second variation of the inverse trace
is a squared norm of a linear map in the coefficient variations.
Comparing it with the negative curvature of
the coefficient penalty bounds the dimension of the space, hence the
coefficient support. The same minimizing representation retains the
inverse moments obtained from its first variations.

For $A\succ0$, this calculation extends to the covariance family
$$
 \Phi_p(A) = \frac{\tr A^{-p}-n}{p}\quad(0<p\le1),
 \quad \Phi_0(A) = -\log\det A.
$$
For unregularized covariances in a symmetric convex hull, the first
variation gives inverse moments of order $2p+1$. The signed second
variation explains the endpoints: at $p = 1$ the positive quadratic
term controls coefficient supports; at $p = 0$ it cancels, leaving a
nonpositive second variation. With a positive penalty parameter,
finite replacements then force each positive-weight column to use
a single generator at a global minimum.
The inverse-trace and determinant optimizations thus have
different sparsity mechanisms within the same family. The
fixed-summand construction shares the variational strategy, with two
barriers in place of this family.

\paragraph{Geometric consequences.}
The two constructions give the following results. For every finite
family of log-concave functions $f_i:\R^n\to[0,\infty)$ satisfying
$0<\int\min_i f_i<\infty$, Theorem~\ref{helly:functional} selects a
nonempty set of original functions such that
\begin{align} \label{intro:helly}
 |J|\le A_0n,\quad
 \int\min_{j\in J}f_j\le(C_0\sqrt n)^n\int\min_i f_i,
\end{align}
where $A_0,C_0$ are absolute. Individual integrability is not assumed.
Indicators give the corresponding result for arbitrary convex sets,
without symmetry. The few-faces estimate of \cite{hellyBP} shows
that the order $\sqrt n$ in the $n$th root of the loss is optimal for
every fixed linear selection count $An$, $A\ge2$.

Let $R_\infty(n)$ be the supremum of the Banach-Mazur distance from
an $n$-dimensional real normed space to $\ell_\infty^n$.
Theorem~\ref{cube:main} proves the lower estimate in
\begin{align} \label{intro:cube}
 \frac{cn^{2/3}}{(\log n)^{5/6}}
 \le R_\infty(n)\le2n^{2/3}\quad(n\ge2),
\end{align}
where the upper bound is due to Friedland-Naor-Youssef
\cite{cube-fny}. The lower bound holds with probability at least
$1-Ce^{-cn}$ for the polar of
$\conv\{\pm g_1,\ldots,\pm g_{2n}\}$, where the $g_i$ are independent
standard Gaussian vectors in $\R^n$. Thus
$\log R_\infty(n)/\log n\to2/3$, determining the polynomial exponent
in Gr\"unbaum's 1960 problem
\cite[Section~6, Remark~(ii)]{cube-grunbaum}, with a remaining
logarithmic gap. The preceding lower exponent was
$5/8$ \cite{cube-friedland58,cube-hmadi}.

The determinant endpoint identifies the optimization in the known
upper bound. Appendix~\ref{cube:upper-appendix} combines its
finite-replacement estimate for the ellipsoidal radius with the
covariance-completion lemma of \cite{cube-fny}. Each bound also needs
information beyond these variational inputs. The lower bound uses
Gaussian estimates; the upper bound uses fixed total weight,
covariance regularization, and a normed-space completion argument.

\subsection{Exact equations and contact geometry}

Let $B_1,\ldots,B_N\succeq0$, $c_i\ge0$, $B = \sum_i c_iB_i$, and $r = \rank B$.
For a matrix $L$ with columns $L_i$, write $L_S$ for its restriction
to the columns indexed by $S$. Theorem~\ref{conic:psd} gives
nonnegative weights $w_i$, with $S = \supp w\subseteq\supp c$, such that
$$
 Lw = Lc,\quad
 (1-\varepsilon)B\preceq\sum_iw_iB_i\preceq(1+\varepsilon)B
 \quad(0<\varepsilon\le1),
$$
and
$$
 |S|-\rank L_S\le
 \min\left\{\dim\spanop\{B_i:i\in S\},
 \left\lceil\frac{Cr}{\varepsilon^2}\right\rceil\right\}.
$$
The augmented family $(B_i,L_i)_{i\in S}$ can be chosen linearly
independent. The quantity on the left is the dimension of coefficient
changes preserving the equations on the surviving support; deleting
coordinates can lower both the support size and the constraint rank.

The approximation scale $O(r/\varepsilon^2)$ is classical
\cite{conicBSS,frameworkCHS}, and constrained discrepancy methods
already provide the matrix input for the Helly application
\cite{conicLWZ}. Here a radial first variation bounds the nuclear norm
of the barrier curvature in a John metric on the constraint kernel.
The second variation forces independent coordinate restrictions to
saturate, reducing the nullity. The exact equations cost their rank
on the final support. The argument also applies to a proper cone
with a $\theta$-self-concordant barrier and an interior target sum,
with $\theta$ replacing $r$; see
Theorem~\ref{conic:barrier-sparsification}.

For the contact application, let $J_1(f)$ be the largest-integral
ellipsoidal minorant of $f$ among positive multiples and affine images
of $h(x) = (1-|x|^2)_+^{1/2}$, and normalize
$J_1(\min_i f_i) = h$. The selected contact weights preserve the
barycenter and vertical moment. If their horizontal covariance is
$M\succ0$, the contact inequalities imply
$$
 \log(\alpha\det A)\le\mathcal D(M),
 \quad \mathcal D(M) = \tr M-\log\det M-n,
$$
for every competing minorant with $\alpha>0$ and $A\succ0$,
$\alpha h(A^{-1}(\cdot-a))\le\min_{j\in J}f_j$;
see Proposition~\ref{helly:determinant}.
The exact moments cancel translation and normalize height, leaving
the determinant defect to control the integral loss.
Preserving $\tr M = n$ and imposing
$(1-\eta)I\preceq M\preceq(1+\eta)I$ gives
$$
 1\le
 \frac{\int J_1(\min_{j\in J}f_j)}{\int J_1(\min_i f_i)}
 \le(1-\eta^2)^{-n/2}
$$
with at most $n+2+\lceil Cn/\eta^2\rceil$ selected functions.
The functional John integral-ratio estimate then yields
\eqref{intro:helly}. A truncation argument establishes integrability
of the selected minimum. Section~\ref{helly:chapter} compares this
linear selection count with minimal Helly counts and with the
functional results of Ivanov-Nasz\'odi \cite{hellyIN22,hellyIN26}.

\subsection{Simultaneous inverse moments and coefficient bounds}

Let $K = \conv\{y_1,\ldots,y_N\} = -K$ span $\R^n$. For an inscribed
basis $X$, put $\beta(X) = \sum_{i = 1}^n\|X^{-T}e_i\|_2$.
Theorem~\ref{framework:square} gives an inscribed basis with columns
$x_i = \sum_jp_{ji}y_j$, coefficient columns $p_i$ in the probability
simplex $\Delta_N$, and positive weights $W_i$. With
$M = (\sum_iW_ix_ix_i^T)^{-1}$, the same representation satisfies
$$
 \beta(X)\le\tr M\le2\inf_{X_0}\beta(X_0),
 \quad \sup_{y\in K}y^TM^3y\le16\tr M,
$$
and $\sum_iW_i\le\tr M$. The infimum is over inscribed bases.
The shifted entropy penalty has finite derivatives even at zero
coefficients, so the first-order tests include every generator.
Squaring and summing these tests produces the third inverse moment.

The coefficient estimates hold for every set $I$ of generator rows
simultaneously. Write
$$
 \rho_I = \sum_{j\in I}y_j^TMy_j,\quad
 e_I = \dim\spanop\{y_j:j\in I\},
$$
and let $\mu_1\ge\cdots\ge\mu_n$ be the eigenvalues of $M$.
The number $k_I$ of positive coefficients in these rows satisfies
$$
 k_I\le n+e_I^2+
 C\left(\rho_I\log N\sum_{\ell = 1}^n\ell\mu_\ell\right)^{1/2}.
$$
The term $n+e_I^2$ bounds the number of linear restrictions for the
coefficient-column sums and nilpotence. Approximation of the tensor
curvature supplies the remaining term. A separate columnwise bound
uses the largest value of $y_j^TMy_j$ in $I$. These estimates count
coefficients representing the $n$ retained basis vectors, and allow
$I$ to be chosen after the minimum is known.

Section~\ref{framework:section} proves the inverse-trace theorem
before the general covariance theorem,
Theorem~\ref{framework:covariance}.
The spectral criteria themselves are classical in optimal design
\cite{frameworkPukelsheim,frameworkKW}; optimization over coordinate
systems and covariance representations also precedes this work
\cite{cube-fny}. The contribution here is the simultaneous control
of inverse moments and representation coefficients through the
penalty and the variation spaces. The functional Helly theorem
uses the fixed-summand construction, rather than a specialization
of the moving-covariance theorem.

\subsection{Gaussian perturbations and the cube argument}

The optimizing basis in the cube argument depends on the entire
Gaussian sample. The proof separates generators with large values of
$g_j^TMg_j$, called exceptional, from the remaining, ordinary generators.
A probabilistic estimate must hold uniformly over the resulting
corrections. After discretizing ordinary coefficients, the relevant
matrices have the form
$$
 Y = ZT+F+\Delta,
$$
where $T\succeq0$ is fixed, $Z$ is standard Gaussian,
$\rank F\le d$, and $\Delta$ is a small rounding error.
The matrices $F$ and $\Delta$ may depend on all of $Z$.
For $n\ge2$, $t>0$ and $1\le d\le n$,
Theorem~\ref{cube:completion} gives one event of
probability at least $1-Ce^{-cd^2}$ on which, whenever
$\|\Delta\|_{\mathrm{op}}\le(4\sqrt t)^{-1}$ and
$\|Y^{-1}\|_{\F}^2\le t$,
$$
 \tr \left(\max(T,t^{-1/2}I)^{-2}\right)\le Cdt,
 \quad
 \sum_{\ell = 1}^n\ell \sigma_\ell(Y^{-1})^2
 \le Cdt\log(en/d).
$$
Here $\sigma_\ell$ denotes decreasing singular values. The maximum
raises every eigenvalue of $T$ below $t^{-1/2}$ to $t^{-1/2}$.
The second estimate is the
weighted spectral sum required by the collective coefficient bound.
The theorem also gives individual inverse singular-value estimates
and a nuclear-norm bound. Proposition~\ref{cube:completion-sharpness}
shows that the logarithmic factors in the weighted spectral and
nuclear-norm estimates are necessary for arbitrary low-rank
perturbations. If additionally $\tr T^2\le t$ and
$t\ge\sqrt n$, the trace estimate implies $dt^2\ge cn^2$.

The conditional structure is equally useful. For a fixed coefficient
record, condition on the exceptional Gaussian columns and use their
span as the first coordinate block. Polar factorization of the
ordinary coefficient matrix gives
$$
 Y = \begin{pmatrix}RA+HT\\GT\end{pmatrix}+\Delta.
$$
Here $R$ represents the conditioned exceptional columns in their span,
$A$ contains their weighted coefficients, and $H,G$ are independent
standard Gaussian blocks under this conditioning.
The exceptional term occurs only in the first block.
A right test frame can be chosen from $G$, leaving $H$ available for
a conditional small-ball estimate. Choices that may depend on the
whole sample are covered by finite nets.

The rank obstruction excludes a small exceptional set. For a larger
set, the weighted spectral estimate saves a logarithmic factor
over the available squared nuclear-norm bound, improving the
collective coefficient estimate, while the third inverse moment
supplies the left test directions. The adaptive threshold and
conditional Gaussian estimates in Section~\ref{cube:chapter}
lead to the limiting balance
$t^3(\log n)^{5/2}\asymp n^2$, which explains the exponent $5/6$
in \eqref{intro:cube}.

\paragraph{Organization and notation.}
Section~\ref{principle:section} gives the singular-value comparison
used below. Sections~\ref{matrix:section}
and~\ref{helly:chapter} treat fixed summands and functional Helly
selection. Section~\ref{framework:section} develops the moving-vector
construction and its spectral endpoints.
Section~\ref{cube:completion-section} establishes the Gaussian
perturbation theorem, and Section~\ref{cube:chapter} combines it
with the inverse-moment and coefficient estimates to prove the
cube lower bound. The appendices contain the simultaneous Gaussian
event and the determinant argument for the known upper bound.
All matrices and spaces are real. All logarithms are natural.
We write $[N] = \{1,\ldots,N\}$, $\supp w = \{i:w_i\ne0\}$, and use
$\preceq$ for the order on real symmetric matrices.
The symbols $\|\cdot\|_{\mathrm F}$ and $\|\cdot\|_*$ denote the
Frobenius and nuclear norms. We write $\sigma_j$ for singular values
in decreasing order and $s_j^\uparrow$ when increasing order is needed.
Constants denoted by $c,C>0$ are absolute and may change from one
occurrence to the next.

\section{Curvature and the dimension of feasible coefficients} \label{principle:section}

Let $z$ be a local minimum of a twice differentiable objective
$\mathcal J$, and let $W$ be a linear space of directions for which
$z+s\xi$ is feasible for sufficiently small $|s|$. Suppose that on $W$,
$$
 D^2\mathcal J(z)[\xi,\xi]
 = \|T\xi\|^2-\|N\xi\|^2.
$$
Second-order optimality gives $\|N\xi\|\le\|T\xi\|$ on $W$.
If $N^*N\succeq\kappa I$ with $\kappa>0$, then $T$ must have at least
$\dim W$ singular values of size $\sqrt\kappa$ or larger. The following
lemma records this comparison, including the dimension lost to linear
restrictions and the effect of approximating $T$ by a low-rank map.

All spaces are finite-dimensional real inner-product spaces.
Singular values are arranged in decreasing order and extended by
zero. We use the nuclear and Hilbert-Schmidt norms
$$
 \|T\|_* = \sum_j s_j(T),\quad
 \|T\|_{\mathrm{HS}}^2 = \sum_j s_j(T)^2.
$$

\begin{lemma}[Singular-value comparison] \label{framework:master}
Let $L,H_+,H_-$ be finite-dimensional real inner-product spaces,
let $b$ be a nonnegative integer, let $W\subseteq L$ have codimension
at most $b$, and let
$T:L\to H_+$ and $N:L\to H_-$ be linear maps such that
$$
 \|N\xi\|\le\|T\xi\|\quad(\xi\in W).
$$
Then there are a contraction $C:H_+\to H_-$ and a linear map
$E:L\to H_-$ such that
\begin{align} \label{framework:eq-factor-master}
 N = CT+E,\quad E|_W = 0,\quad \rank E\le b.
\end{align}
For every nonnegative integer $k$ and every map $T_0:L\to H_+$
of rank at most $k$,
\begin{align} \label{framework:eq-sing-master}
 s_{j+b+k}(N)\le s_j(T-T_0)\quad(j\ge1).
\end{align}
If $N^{*}N\succeq\kappa I_L$ with $\kappa>0$, then
\begin{align} \label{framework:eq-count-master}
 \dim L\le b+\#\{j:s_j(T)\ge\sqrt\kappa\},
\end{align}
and
\begin{align} \label{framework:eq-tail-master}
 (\dim L-b-k)_+\kappa\le\|T-T_0\|_{\mathrm{HS}}^2.
\end{align}
In particular,
\begin{align} \label{framework:eq-dimension-forms}
 \dim L\le b+\min\left\{
 \rank T,\frac{\|T\|_{*}}{\sqrt\kappa},
 \frac{\|T\|_{\mathrm{HS}}^2}{\kappa}\right\}.
\end{align}
\end{lemma}

\begin{proof}
Define $C_0(T\xi) = N\xi$ for $\xi\in W$.
If $T\xi = 0$, the preceding inequality gives $N\xi = 0$, so $C_0$
is well-defined and contractive. Extend it by zero on $T(W)^\perp$
to a contraction $C$, and set $E = N-CT$. This map vanishes on $W$,
which proves \eqref{framework:eq-factor-master}.
The map $CT_0+E$ has rank at most $b+k$, while
$N-(CT_0+E) = C(T-T_0)$. Singular-value approximation and contraction
by $C$ give \eqref{framework:eq-sing-master}.
If $N^{*}N\succeq\kappa I$, the right singular subspace of $T$
corresponding to singular values below $\sqrt\kappa$ has trivial
intersection with $W$. This proves
\eqref{framework:eq-count-master}. Alternatively, squaring and summing
\eqref{framework:eq-sing-master} for
$1\le j\le(\dim L-b-k)_+$ gives \eqref{framework:eq-tail-master}.
The rank, nuclear, and Hilbert-Schmidt estimates follow.
\end{proof}

All adjoints and norms refer to the chosen inner products. In
Section~\ref{matrix:section}, a metric on the kernel of the exact equations
allows the lemma to bound the dimension of unsaturated variations.
In Section~\ref{framework:section}, linear restrictions depending on the
optimizing basis allow it to bound the number of positive coefficients.

\section{Spectral sparsification with linear constraints} \label{matrix:section}

We first keep the positive semidefinite summands fixed and vary their
nonnegative coefficients subject to $Lw = Lc$. For $S\subseteq[N]$,
let $L_S$ be the submatrix of $L$ with columns indexed by $S$. On a
positive support $S$, the feasible coefficient variations form a space of dimension
$\dim\ker L_S = |S|-\rank L_S$. This is the quantity controlled below;
deleting coordinates can decrease both the support size and the
constraint rank.

\begin{theorem}[Spectral approximation with linear constraints]
 \label{conic:psd}
There is an absolute integer $C>0$ with the following property.
Let $B_1,\ldots,B_N$ be real symmetric positive semidefinite matrices,
let $c\in[0,\infty)^N$, and set
$B = \sum_i c_iB_i$ and $r = \operatorname{rank}B$.
For any real matrix $L$ with $N$ columns and any $0<\varepsilon\le1$,
there are weights $w\in[0,\infty)^N$, with
$S = \{i:w_i>0\}\subseteq\{i:c_i>0\}$, such that
\begin{align}
 Lw& = Lc, \label{matrix:exact} \\
 (1-\varepsilon)B&\preceq\sum_iw_iB_i
 \preceq(1+\varepsilon)B, \label{matrix:approx} \\
 \dim\ker L_S&\le
 \min\left\{\dim\operatorname{span}\{B_i:i\in S\},
 \left\lceil\frac{{C}r}{\varepsilon^2}\right\rceil\right\}.
 \label{matrix:nullity}
\end{align}
Moreover, the augmented vectors $(B_i,L_i)$, $i\in S$, may be chosen
linearly independent, where $L_i$ is the $i$th column of $L$.
In particular,
$$
 |S|\le\operatorname{rank}L_S+
 \left\lceil\frac{{C}r}{\varepsilon^2}\right\rceil .
$$
Singular sums and zero matrices are allowed.
\end{theorem}

The constant is not optimized. The proof is existential: it uses a
global minimum of a nonconvex objective.

\paragraph{Proof strategy.}
Write a coefficient change as $w_i\mapsto w_i(1+z_i)$, with $z$ in
the kernel of the rescaled constraints and the box $[-1,1]^S$.
In a metric adapted to this kernel, minimize two spectral barriers
minus a quadratic penalty.
A radial first variation bounds the nuclear norm of the barrier's
curvature operator. Second-order optimality then forces a fixed fraction
of the independent coordinate restrictions to be saturated. Evenness
allows enough of these coordinates to equal $-1$, so they can be deleted.
Iteration reduces the nullity geometrically; summing the errors backwards
from the last step controls the total spectral distortion.

\paragraph{Relation to earlier sparsification results.}
Rudelson \cite{frameworkRudelson} obtained approximate identity
decompositions with $O_\varepsilon(r\log r)$ points by sampling.
For rank-one sums, Batson, Spielman and Srivastava \cite{conicBSS}
obtained $O(r/\varepsilon^2)$ nonzero weights.
The extension to positive semidefinite summands of arbitrary rank is
due to de~Carli~Silva, Harvey and Sato \cite{frameworkCHS}. Their work also
preserves total coefficient mass exactly and additional nonnegative
costs approximately.

Barycenter conditions in geometric applications have been treated by
Srivastava \cite{frameworkSrivastava} and Friedland-Youssef
\cite{frameworkFY}; the latter obtain spectral approximation after a
small translation of the selected points. Partial-coloring and
discrepancy methods provide a framework for retaining exact equations
during sparsification. Relevant developments include Reis-Rothvoss
\cite{frameworkRR}, Jambulapati-Reis-Tian \cite{frameworkJRT}, and,
for the arbitrary linear constraints needed here, Lau-Wang-Zhou
\cite[Theorem~1.2]{conicLWZ}. In particular, the last result already
supplies the matrix input required for Section~\ref{helly:chapter}.
Theorem~\ref{conic:psd} records a bound with coefficient one on
$\rank L_S$, evaluated on the final support, and obtains it by the
variational argument above.

Optimization methods also enter the sparsification arguments of
Allen-Zhu-Liao-Orecchia \cite{frameworkALO} and the discrepancy
algorithms of Pesenti-Vladu \cite{frameworkPV}.
Lau-Wang-Zhou \cite[Sections~1.2 and~3.2-3.4]{conicLWZ} also use
first- and second-order variation; the present argument combines a
global penalized minimum with a John metric on the constraint kernel
to control the nullity $|S|-\rank L_S$ on the surviving support.
The global minimization used below does not provide an efficient
algorithm.

\subsection{Elimination and a coefficient metric}

The first lemma records conic Carath\'eodory elimination and the exact
change in nullity under deletion. The second chooses a metric controlling
both the coordinate functionals and the radius of the feasible cube section.

\begin{lemma} \label{matrix:elimination}
A nonnegative combination of finitely many vectors has a nonnegative
representation on an independent subset of its original support.
If $S'\subseteq S$ and
$J = S\setminus S'$, then
\begin{align} \label{matrix:deletion}
 \dim\ker L_{S'} = \dim\ker L_S-
 \operatorname{rank}(R_J:\ker L_S\longrightarrow\mathbb R^J),
\end{align}
where $R_J$ is coordinate restriction. The same identity holds after
any invertible diagonal rescaling of the coordinates.
\end{lemma}

\begin{proof}
Given a dependent positive representation $\sum_{i\in S}a_iv_i$,
choose $u\ne0$ with $\sum_{i\in S}u_iv_i = 0$ and some $u_i>0$.
Replacing $a_i$ by $a_i-tu_i$, where
$t = \min_{u_i>0}a_i/u_i$, preserves nonnegativity and the represented
vector and removes an index. Iteration gives the independent support.
Extension by zero identifies $\ker L_{S'}$ with the kernel
of $R_J$ on $\ker L_S$. Rank-nullity proves
\eqref{matrix:deletion}. Diagonal rescaling preserves every coordinate
vanishing condition.
\end{proof}

The ambient Euclidean metric can give the wrong dimension dependence:
if $H = \operatorname{span}\{(1,\ldots,1)\}\subseteq\R^m$, then
$\dim H = 1$, but $H\cap[-1,1]^m$ has Euclidean radius $\sqrt m$.
In the metric below its radius is one. In general, the squared radius
is at most $\dim H$, while the coordinate functionals have dual norm
at most one. These bounds enter the first and second variations,
respectively.

\begin{lemma}[John's theorem for a section of the cube] \label{matrix:metric}
Let $H\subseteq\mathbb R^m$ have dimension $h>0$.
There is an inner product on $H$, with norm $\|\cdot\|_H$, for which
each coordinate functional $\ell_i(z) = z_i$ has dual norm at most one,
and
$$
 \|z\|_H^2\le h\quad(z\in H\cap[-1,1]^m).
$$
\end{lemma}

\begin{proof}
Let $\mathcal E$ be the maximum-volume ellipsoid contained in
$K = H\cap[-1,1]^m$, and take $\mathcal E$ as the unit ball of
$\|\cdot\|_H$. The symmetric form of John's theorem \cite{john}
gives $\mathcal E\subseteq K\subseteq\sqrt h \mathcal E$.
The first inclusion bounds the dual norm of every coordinate
functional by one. The second gives the asserted bound on $K$.
\end{proof}

\subsection{The two-sided barrier and its resolvent}

The radial derivative of a two-sided logarithmic barrier controls its
resolvent traces. Positivity of the summands then bounds the nuclear
norm of the map representing the Hessian as a squared norm.

For $\delta>0$ and a symmetric $r\times r$ matrix $Z$ with
$-\delta I\prec Z\prec\delta I$, define
\begin{align} \label{matrix:barrier}
 \Psi_\delta(Z)
 = -\log\det(\delta I-Z)-\log\det(\delta I+Z)
 = -\log\det(\delta^2I-Z^2).
\end{align}
Put $R_- = (\delta I-Z)^{-1}$ and $R_+ = (\delta I+Z)^{-1}$.
Differentiating the determinants gives
\begin{align}
 D\Psi_\delta(Z)[H]
 & = \operatorname{tr}((R_- - R_+)H)
 = 2\operatorname{tr}(Z(\delta^2I-Z^2)^{-1}H),
 \label{matrix:first-variation} \\
 D\Psi_\delta(Z)[Z]
 & = 2\operatorname{tr}(Z^2(\delta^2I-Z^2)^{-1})
 = \delta\operatorname{tr}(R_-+R_+)-2r,
 \label{matrix:radial} \\
 D^2\Psi_\delta(Z)[H,H]
 & = \operatorname{tr}(R_-HR_-H)+\operatorname{tr}(R_+HR_+H).
 \label{matrix:hessian}
\end{align}
These formulas follow from $D\log\det X[H] = \tr(X^{-1}H)$
and differentiation of $(\delta I\mp Z)R_\mp = I$; in particular,
\eqref{matrix:hessian} does not require $H$ to commute with $Z$.
The resolvent traces are the upper and lower trace potentials used in
BSS \cite{conicBSS}. Here their sum is bounded by an inward radial
variation.

\subsection{A partial step with exact equations}

\begin{lemma} \label{matrix:partial}
Let $r\ge1$ and let $A_1,\ldots,A_m\succeq0$ be $r\times r$ matrices with
$E = \sum_iA_i\succ0$, and let $H\subseteq\mathbb R^m$ have dimension
$h>0$. Set $\delta = 8\sqrt{r/h}$.
There is $z\in H\cap[-1,1]^m$ such that
$$
 \delta E\pm\sum_i z_iA_i\succ0,
 \quad
 \operatorname{rank}(R_{J_-}|_H)\ge h/4,
 \quad J_- = \{i:z_i = -1\}.
$$
\end{lemma}

\begin{proof}
Congruence by $E^{-1/2}$ reduces to $E = I_r$.
Choose the metric of Lemma~\ref{matrix:metric}, set $\lambda = r/h$,
and write $Z(z) = \sum_i z_iA_i$.
Minimize
\begin{align} \label{matrix:objective}
 \mathcal J(z) = \Psi_\delta(Z(z))-\lambda\|z\|_H^2
 \quad\text{ on }H\cap[-1,1]^m,
\end{align}
assigning value $+\infty$ when either barrier argument is not positive
definite. Zero is feasible. On the compact cube section the barrier
arguments have bounded eigenvalues, so each negative log determinant
is bounded below. At a singular boundary one tends to $+\infty$;
the quadratic term is bounded. A minimizing sequence therefore has
a feasible limit of finite value, which is a minimizer.

At the minimum, the inward variation $(1-t)z$, $0\le t\le1$,
is feasible even when some coordinates are saturated. Thus
$D\mathcal J(z)[z]\le0$. With $R_\pm$ evaluated at $Z(z)$,
\eqref{matrix:radial} gives
\begin{align} \label{matrix:resolvent-budget}
 \delta\operatorname{tr}(R_-+R_+)-2r
 \le2\lambda\|z\|_H^2\le2r,
 \quad \operatorname{tr}(R_-+R_+)\le\frac{4r}{\delta}.
\end{align}
Define a linear map from $H$ to two copies of the symmetric matrices,
with the Frobenius inner products, by
$$
 T\xi = \left(R_-^{1/2}(\sum_i\xi_iA_i)R_-^{1/2},
 R_+^{1/2}(\sum_i\xi_iA_i)R_+^{1/2}\right).
$$
The Hessian of the barrier is $\|T\xi\|^2$. To bound the number of
its large singular values, expand $T$ as a sum of coordinate rank-one
operators.
Each coordinate functional has dual norm at most one, and
$\|P\|_{\mathrm F}\le\operatorname{tr}P$ for $P\succeq0$.
The triangle inequality and the rank-one nuclear-norm formula yield
\begin{align}
 \|T\|_{*}
 &\le\sum_i\left(
 \|R_-^{1/2}A_iR_-^{1/2}\|_{\mathrm F}
 +\|R_+^{1/2}A_iR_+^{1/2}\|_{\mathrm F}\right)\notag \\
 &\le\sum_i\operatorname{tr}((R_-+R_+)A_i)
 = \operatorname{tr}(R_-+R_+)\le4r/\delta.
 \label{matrix:nuclear-budget}
\end{align}

Let $J_{\mathrm{sat}} = \{i:|z_i| = 1\}$ and
$W = \ker(R_{J_{\mathrm{sat}}}|_H)$.
Every $\xi\in W$ gives a two-sided feasible straight variation
$z+s\xi$ for sufficiently small $|s|$.
The Hessian identity~\eqref{matrix:hessian} and local minimality imply
$$
 0\le D^2\mathcal J(z)[\xi,\xi]
 = \|T\xi\|^2-2\lambda\|\xi\|_H^2.
$$
The nuclear-norm bound in Lemma~\ref{framework:master} therefore gives
$$
 \dim W\le\frac{\|T\|_{*}}{\sqrt{2\lambda}}
 \le\frac{4r/\delta}{\sqrt{2r/h}}
 = \frac{h}{2\sqrt2}<\frac h2.
$$
Equivalently, more than $h/2$ independent coordinates are saturated.
Since the restriction rank on their union is at most the sum of the
restriction ranks on $J_-$ and $J_+ = \{i:z_i = 1\}$, one of these ranks
is at least $h/4$. Both the objective and its domain are even;
replace $z$ by $-z$ if necessary to obtain the asserted sign.
\end{proof}

\subsection{Iteration and the final support}

By \eqref{matrix:deletion}, the rank of the deleted coordinate
restrictions is exactly the decrease in feasible dimension.

\begin{proof}[Proof of Theorem~\ref{conic:psd}]
Discard indices with $c_i = 0$. If $x\in\ker B$, then
$0 = \sum_i c_i x^TB_ix$ forces $x^TB_ix = 0$ for every remaining
index, hence $B_ix = 0$ by positive semidefiniteness.
Thus every remaining matrix is supported on $\operatorname{range}B$.
If $r = 0$, all these matrices are zero, and applying
Lemma~\ref{matrix:elimination} to the columns of $L$ gives an exact
representation on independent columns, with $\dim\ker L_S = 0$.
Assume henceforth $r>0$ and restrict to $\operatorname{range}B$,
where $B\succ0$; this restriction preserves all relevant spans.

Fix an absolute integer $C\ge2^{14}$, put
$D = \lceil{C}r/\varepsilon^2\rceil$, and start with $w = c$.
At a stage with positive support $S$, current sum
$E_w = \sum_{i\in S}w_iB_i\succ0$, and
$$
 H = \ker(L_S\operatorname{diag}(w_i:i\in S)),
 \quad h = \dim H = \dim\ker L_S,
$$
stop if $h\le D$.
Otherwise apply Lemma~\ref{matrix:partial} to the matrices $w_iB_i$
and replace $w_i$ by $w'_i = w_i(1+z_i)$.
The equations $Lw' = Lw$ hold exactly, and
$$
 (1-\delta)E_w\preceq E_{w'}\preceq(1+\delta)E_w,
 \quad \delta = 8\sqrt{r/h}<\varepsilon/16<1.
$$
In particular the new sum remains positive definite.
The deleted coordinates are precisely $J_-$.
Lemma~\ref{matrix:elimination}, with the current diagonal rescaling,
gives the exact nullity decrease
\begin{align} \label{matrix:nullity-decrease}
 \dim\ker L_{S\setminus J_-}
 = h-\operatorname{rank}(R_{J_-}|_H)\le3h/4.
\end{align}
The iteration terminates with final nullity at most $D$.

Let $h_1,\ldots,h_q>D$ be the nullities before the steps performed.
By \eqref{matrix:nullity-decrease},
$h_j>(4/3)^{q-j}D$, and consequently
$$
 \sum_{j = 1}^q\delta_j
 <\frac8{1-\sqrt3/2}\sqrt{\frac rD}
 <64\sqrt{\frac rD}\le\frac\varepsilon2.
$$
The claim is immediate if no step is performed.
Iterating the spectral inequalities and using
$$
 \prod_j(1-\delta_j)\ge1-\sum_j\delta_j,
 \quad
 \prod_j(1+\delta_j)\le e^{\sum_j\delta_j}
 \le e^{\varepsilon/2}\le1+\varepsilon
$$
proves \eqref{matrix:approx}; the last inequality follows from
$\log(1+t)\ge t/2$ for $0\le t\le1$.
The equations \eqref{matrix:exact} hold throughout.

Finally apply Lemma~\ref{matrix:elimination} to the augmented
vectors $(B_i,L_i)$ on the current support.
This preserves the actual approximating sum and all exact moments,
and produces an independent augmented family on a smaller support.
Its nullity cannot increase, by \eqref{matrix:deletion}, and
$$
 |S| = \dim\operatorname{span}\{(B_i,L_i):i\in S\}
 \le\dim\operatorname{span}\{B_i:i\in S\}
 +\operatorname{rank}L_S.
$$
This proves both simultaneous bounds in \eqref{matrix:nullity}
and augmented independence. Extending the restricted matrices by zero
on $\ker B$ gives the asserted inequalities in the original space.
\end{proof}

\subsection{Cones with self-concordant barriers}
 \label{conic:barrier-section}

The derivative inequality underlying the preceding argument also holds
for a self-concordant barrier on a proper cone.

For convex cones, Saunderson \cite{frameworkSaunderson} proves support
bounds in terms of a logarithmically homogeneous self-concordant
barrier parameter, with the linear order under an additional pairwise
self-concordance condition. Liu \cite[Theorem~1.1]{frameworkLiu} obtains
$(1+\delta)\rank L+C_\delta\theta/\varepsilon^2$ for arbitrary exact
equations and a $\theta$-self-concordant barrier, for each fixed
$\delta>0$, together with a randomized construction in an oracle model.
Theorem~\ref{conic:barrier-sparsification} gives an existential
variational proof with coefficient one on $\rank L_S$, evaluated on
the final support. The linear order in $\theta$ and arbitrary exact
equations are already supplied by Liu's theorem, whose rank term
after zero initial coefficients are discarded is $\rank L_{\supp c}$.

Let $V$ be a finite-dimensional Euclidean space of positive dimension.
A proper cone $K\subseteq V$ is closed, convex, has nonempty interior,
and satisfies $K\cap(-K) = \{0\}$. Write $x\preceq_K y$ when
$y-x\in K$. For $\theta>0$, a $\theta$-self-concordant barrier on $K$ is a convex
$C^3$ function $\phi:\operatorname{int}K\to\R$ with positive definite
Hessian, tending to $+\infty$ at the boundary, such that
$$
 |D^3\phi(x)[u,u,u]|\le2(D^2\phi(x)[u,u])^{3/2},\quad
 |D\phi(x)[u]|^2\le\theta D^2\phi(x)[u,u]
$$
for $x\in\operatorname{int}K$ and $u\in V$.

\begin{lemma} \label{conic:barrier-derivatives}
For such a barrier, $x\in\operatorname{int}K$ and $v\in K$,
\begin{align} \label{conic:positive-derivatives}
 \sqrt{D^2\phi(x)[v,v]}\le-D\phi(x)[v],\quad
 -D\phi(x)[x]\le\theta.
\end{align}
\end{lemma}

\begin{proof}
These are \cite[Lemma~2.1]{frameworkLiu}; we recall the argument.
For $v\ne0$, put $g(t) = \phi(x+tv)$, $t\ge0$.
If $g'(t)>0$ anywhere, convexity and
$(1/g')'\le-1/\theta$ would force $g'$ to blow up in finite time.
Thus $g'\le0$. Self-concordance gives
$$
 g''(t)\ge\frac{g''(0)}{(1+t\sqrt{g''(0)})^2}.
$$
Integrating over $[0,\infty)$ yields $-g'(0)\ge\sqrt{g''(0)}$.
The assertion is immediate for $v = 0$.
For the radial estimate, put $q(t) = -D\phi(tx)[x]>0$ for $t>0$.
The barrier parameter inequality gives
$q'(t)\le-q(t)^2/\theta$, hence $(1/q)'\ge1/\theta$.
Integrating from $s$ to $t$ and letting $s\downarrow0$ gives
$q(t)\le\theta/t$. Take $t = 1$.
\end{proof}

\begin{theorem}[Conic approximation with linear constraints]
 \label{conic:barrier-sparsification}
Suppose that the proper cone $K$ admits a $\theta$-self-concordant
barrier. Let $v_1,\ldots,v_N\in K$, let $c\in[0,\infty)^N$, and
assume that $e = \sum_i c_iv_i\in\operatorname{int}K$.
For every real matrix $L$ with $N$ columns and $0<\varepsilon\le1$,
there are $w\in[0,\infty)^N$ and
$S = \{i:w_i>0\}\subseteq\{i:c_i>0\}$ such that
$$
 Lw = Lc,\quad
 (1-\varepsilon)e\preceq_K\sum_iw_iv_i
 \preceq_K(1+\varepsilon)e,
$$
and
$$
 \dim\ker L_S\le
 \min\left\{\dim\operatorname{span}\{v_i:i\in S\},
 \left\lceil\frac{C\theta}{\varepsilon^2}\right\rceil\right\}.
$$
The augmented vectors $(v_i,L_i)$, $i\in S$, may be chosen linearly
independent. Here one may use the same constant $C\ge2^{14}$ as in
Theorem~\ref{conic:psd}.
\end{theorem}

\begin{proof}
We establish the analogue of Lemma~\ref{matrix:partial}.
Let $a_1,\ldots,a_m\in K$ have sum $E\in\operatorname{int}K$,
and let $H\subseteq\R^m$ have dimension $h>0$.
Use the metric of Lemma~\ref{matrix:metric}, put
$Z(z) = \sum_i z_i a_i$, $\lambda = \theta/h$ and
$\delta = 8\sqrt{\theta/h}$, and minimize
$$
 \mathcal J(z) = \phi(\delta E-Z(z))+\phi(\delta E+Z(z))
 -\lambda\|z\|_H^2
 \quad\text{on }H\cap[-1,1]^m.
$$
The value is $+\infty$ unless both barrier arguments are interior.
Zero is feasible. A supporting affine function bounds $\phi$ below
on the bounded sets of possible arguments; approaching the boundary
makes the relevant barrier term tend to $+\infty$. The quadratic
term is bounded. Compactness therefore gives a finite minimizer.

At that minimum write $x_\pm = \delta E\pm Z(z)$ and
$$
 P = -D\phi(x_-)[E]-D\phi(x_+)[E].
$$
The inward radial variation gives
$$
 \delta P+D\phi(x_-)[x_-]+D\phi(x_+)[x_+]
 \le2\lambda\|z\|_H^2\le2\theta.
$$
By \eqref{conic:positive-derivatives},
$P\le4\theta/\delta$.
Using the positive square roots of the Hessians, define
$$
 T\xi = \left(
 \nabla^2\phi(x_-)^{1/2}\sum_i\xi_i a_i,
 \nabla^2\phi(x_+)^{1/2}\sum_i\xi_i a_i\right).
$$
The barrier Hessian on $H$ is $\|T\xi\|^2$. Expanding $T$ in
coordinate rank-one operators and applying
\eqref{conic:positive-derivatives} gives
\begin{align*}
 \|T\|_*
 &\le\sum_i\sqrt{D^2\phi(x_-)[a_i,a_i]+D^2\phi(x_+)[a_i,a_i]}\\
 &\le\sum_i\big(-D\phi(x_-)[a_i]-D\phi(x_+)[a_i]\big)
 = P\le4\theta/\delta.
\end{align*}
Here the coordinate functionals have dual norm at most one, and
$\sum_i a_i = E$.

Let $W$ be the subspace of $H$ vanishing on all saturated
coordinates. Second-order optimality gives
$\|T\xi\|^2\ge2\lambda\|\xi\|_H^2$ for $\xi\in W$.
Lemma~\ref{framework:master} now yields
$$
 \dim W\le\frac{4\theta/\delta}{\sqrt{2\theta/h}}
 = \frac{h}{2\sqrt2}<\frac h2.
$$
The saturated coordinate restrictions consequently have rank greater
than $h/2$. As in Lemma~\ref{matrix:partial}, evenness permits a sign
choice for which $\{i:z_i = -1\}$ has restriction rank at least
$h/4$. We have thus proved the conic partial step, with
$\delta E\pm Z(z)\in\operatorname{int}K$.

Discard indices with $c_i = 0$ and iterate as in the proof of
Theorem~\ref{conic:psd}, with $r$ replaced by $\theta$.
At a positive support $S$ with sum $E = \sum_{i\in S}w_iv_i$, apply
the partial step to $a_i = w_iv_i$ and
$H = \ker(L_S\operatorname{diag}(w_i:i\in S))$ while
$h = \dim\ker L_S>D = \lceil C\theta/\varepsilon^2\rceil$.
The update $w_i\mapsto w_i(1+z_i)$ preserves $Lw$, and
\eqref{matrix:nullity-decrease} reduces $h$ to at most $3h/4$.
Since $\delta = 8\sqrt{\theta/h}<\varepsilon/16<1$, the new sum
$E+Z(z) = (1-\delta)E+(\delta E+Z(z))$ remains interior.
The geometric sum and product estimates in that proof give
$$
 \sum_j\delta_j<64\sqrt{\theta/D}\le\varepsilon/2.
$$
Thus the order factors $1\pm\delta_j$ give the asserted approximation,
with final nullity at most $D$. Elimination on the augmented vectors
$(v_i,L_i)$ preserves the sum and exact equations, cannot increase
nullity, and gives independence and the span bound, exactly as in the
last paragraph of that proof.
\end{proof}

The theorem assumes an interior target but does not require logarithmic
homogeneity. For a positive semidefinite target of rank $r>0$,
restriction to its range gives the barrier $-\log\det$ with parameter
$r$, as in Theorem~\ref{conic:psd}.

\section{A quantitative functional Helly theorem} \label{helly:chapter}

Exact barycenter and vertical contact moments remove the translation
term and fix the height coefficient in the contact inequality for every
competing ellipsoidal minorant. The resulting integral bound depends
only on the contact covariance. We prove this implication first, then
sparsify the contact equations and select members of the original family.

For a finite family $f_1,\ldots,f_N:\R^n\to[0,\infty)$ and a
nonempty $J\subseteq[N]$, put $f_J = \min_{j\in J}f_j$ and $f = f_{[N]}$.
All integrals in this section are with respect to Lebesgue measure.
The height function of the Euclidean unit ball in $\R^{n+1}$ is
$$
 h(x) = (1-|x|^2)_+^{1/2}.
$$
An ellipsoidal function has the form
\begin{align} \label{helly:ellipsoidal}
 g(x) = \alpha h(A^{-1}(x-a)),\quad
 A\succ0,\quad\alpha>0,\quad a\in\R^n,
 \quad \int g = \alpha\det A\int h.
\end{align}
Radial symmetry of $h$ allows the matrix $A$ to be symmetric positive
definite. For a log-concave function $G$ with $0<\int G<\infty$,
its functional John $1$-function
$J_1(G)$ is the ellipsoidal function of largest integral below $G$.
The results of Ivanov-Nasz\'odi and the regularization argument below
ensure existence and uniqueness.

\begin{theorem}[Quantitative functional Helly theorem] \label{helly:functional}
Let $f_1,\ldots,f_N:\R^n\to[0,\infty)$ be log-concave functions,
where $n\ge1$, and suppose $0<\int f<\infty$.
There are absolute constants $A_0,C_0>0$ and a nonempty $J\subseteq[N]$
such that
\begin{align} \label{helly:main-bound}
 |J|\le A_0n,\quad
 \int f_J\le(C_0\sqrt n)^n\int f.
\end{align}
More precisely, for every $0<\eta<1$, one can choose $J$ with
\begin{align} \label{helly:count}
 |J|\le n+2+\left\lceil\frac{{C}n}{\eta^2}\right\rceil
\end{align}
such that $0<\int f_J<\infty$ and
\begin{align}
 1\le\frac{\int J_1(f_J)}{\int J_1(f)}
 &\le(1-\eta^2)^{-n/2}, \label{helly:john-ratio} \\
 \frac{\int f_J}{\int f}
 &\le\left(\frac{\Theta_1\sqrt n}{\sqrt{1-\eta^2}}\right)^n.
 \label{helly:integral-ratio}
\end{align}
Here $C$ is the constant in Theorem~\ref{conic:psd}, and $\Theta_1$ is
the absolute functional John integral-ratio constant in
\eqref{helly:classical-ratio}. In particular, one may take
$A_0 = 4C+3$ and $C_0 = 2\Theta_1/\sqrt3$.
\end{theorem}

\paragraph{Relation to earlier results.}
Quantitative versions of Helly's theorem \cite{helly1923} compare the
volume of a full intersection with that of a small subfamily.
The selection count and the dimensional loss are distinct parameters;
see also the discussion in \cite[Section~2]{hellyNas25}.
B\'ar\'any, Katchalski, and Pach \cite{hellyBKP82,hellyBKP84}
proved a volume bound using $2n$ convex sets, and the facet halfspaces
of a cube show that this count is necessary. For this count,
Nasz\'odi \cite[Theorem~1.1]{hellyNas16} obtained $(Cn)^{2n}$, and
Brazitikos \cite[Section~3]{hellyBra17} improved the bound to
$(Cn)^{3n/2}$, also proved in \cite{hellyAHK23}.
With a larger absolute multiple of $n$ halfspaces,
\cite[Theorem~1.4]{hellyBra17} gives $(Cn)^n$; for symmetric strips,
\cite[Theorem~1.2]{hellyBra17} already gives $(C_a\sqrt n)^n$
with at most $an$ strips, for every fixed $a>1$.
The continuous tradeoff in \cite[Theorem~1.3]{hellyFGM} reaches
an $n$th-root loss of order $\sqrt n$ with $O(n^2)$ halfspaces.
Corollary~\ref{helly:volume} obtains this order with $O(n)$ original
members without symmetry, at a larger linear selection count than $2n$.
Proposition~\ref{helly:sharpness} proves that the dimensional order is
sharp for a fixed linear selection count.

For log-concave functions, Ivanov-Nasz\'odi
\cite[Theorem~1.2]{hellyIN22} obtained a selection of $3n+2$ functions
with integral ratio at most $(100n)^{5n/2}$, allowing nonintegrable
individual functions. Their later theorem
\cite[Theorem~2]{hellyIN26} obtains the optimal count $2n+1$ for
integrable input functions, with ratio $\exp(Cn^5)$.
Theorem~\ref{helly:functional} allows a larger linear count and
controls both the selected integral and its functional John integral.

Theorem~\ref{conic:psd} supplies the matrix input; sufficient constrained
spectral approximation is also given by
\cite[Theorem~1.2 and Section~4.1]{conicLWZ} and \cite{frameworkLiu}.
The geometric step is Proposition~\ref{helly:determinant}, which
bounds all ellipsoidal minorants from the preserved contact moments.

\subsection{A determinant bound from contact moments}

For an upper semicontinuous function $q:\R^n\to[0,\infty)$, write
$$
 \mathcal L(q) = \{(x,t):|t|\le q(x)\}
$$
for its closed lift. If $q\ge h$, then $B_2^{n+1}\subseteq\mathcal L(q)$.
The following contact inequality is
\cite[Claim~5.7, equation~(5.11)]{hellyIN22}; we give its proof,
including the case of an equatorial contact.

\begin{lemma}[Contact inequality] \label{helly:contact-inequality}
Let $q:\R^n\to[0,\infty)$ be upper semicontinuous and log-concave,
with $q\ge h$. Suppose $t\ge0$ and $(u,t)$ belongs to the common
boundary of $B_2^{n+1}$ and $\mathcal L(q)$.
Every ellipsoidal function \eqref{helly:ellipsoidal} with $g\le q$
satisfies
\begin{align} \label{helly:contact-bound}
 t^2\log\alpha+\langle u,a\rangle+
 \langle u,(A-I_n)u\rangle\le0.
\end{align}
\end{lemma}

\begin{proof}
Suppose first that $t>0$. Then $|u|<1$ and $q\ge h>0$ in a
neighborhood of $u$, where $\log q$ is finite, concave and continuous.
If $q(u)>t$, the point $(u,t)$ would be interior to $\mathcal L(q)$.
Hence $q(u) = h(u) = t$. The convex function $-\log q$ is bounded above
by the differentiable convex function $-\log h$ near $u$, with equality
at $u$. It has a subgradient there, and every such subgradient equals
$\nabla(-\log h)(u) = u/t^2$: apply the supporting inequality in each
direction and its negative and differentiate the upper bound.
The global subgradient inequality gives the standard supporting
exponential bound \cite[Lemma~3.1]{hellyIN22}
\begin{align} \label{helly:supporting-exponential}
 q(x)\le t\exp\left(-\frac{\langle u,x-u\rangle}{t^2}\right),
\end{align}
also at points where $q = 0$.
Evaluating $g\le q$ at $x = a+Au$ gives
$$
 \alpha t\le
 t\exp\left(-\frac{\langle u,a+Au-u\rangle}{t^2}\right).
$$
Taking logarithms and multiplying by $t^2$ proves
\eqref{helly:contact-bound}.

Now suppose $t = 0$, so $|u| = 1$. The closed convex set
$C_q = \overline{\{x:q(x)>0\}}$ contains $B_2^n$.
One has $u\in\partial C_q$. Indeed, the interior of $C_q$ equals
the interior of the convex positivity set of $q$. At any point of
that interior, $\log q$ is finite and continuous, so a neighborhood
of $(u,0)$ would lie in $\mathcal L(q)$, contrary to the contact
property. A supporting hyperplane of $C_q$ at $u$ also supports
$B_2^n$ there. Its normal is therefore a positive multiple of $u$,
and $\langle u,x\rangle\le1$ for $x\in C_q$.
Since $g\le q$, the open base ellipsoid of $g$ lies in the positivity
set of $q$, and its closure $a+AB_2^n$ lies in $C_q$.
Thus $\langle u,a+Au\rangle\le1 = |u|^2$, proving
\eqref{helly:contact-bound} when $t = 0$.
\end{proof}

The equatorial argument uses the boundary of the lift, rather than
an equality $q(u) = h(u)$; an indicator function may be positive on
the boundary of its support.

\begin{proposition}[Control by exact contact moments] \label{helly:determinant}
Let $q_1,\ldots,q_\ell:\R^n\to[0,\infty)$ be upper semicontinuous
log-concave functions with $q_i\ge h$. For each $i$, let $(u_i,t_i)$,
with $t_i\ge0$, belong to the common boundary of $B_2^{n+1}$ and
$\mathcal L(q_i)$. Suppose nonnegative weights $w_i$ satisfy
$$
 \sum_iw_iu_i = 0,\quad \sum_iw_it_i^2 = 1,\quad
 M = \sum_iw_iu_iu_i^T\succ0.
$$
Put $F = \min_{i:w_i>0}q_i$. Then every ellipsoidal function $g\le F$
satisfies
\begin{align} \label{helly:determinant-general}
 \frac{\int g}{\int h} = \alpha\det A
 \le\frac{\exp(\tr M-n)}{\det M}.
\end{align}
If $(1-\eta)I_n\preceq M\preceq(1+\eta)I_n$ for $0<\eta<1$, then
\begin{align} \label{helly:quadratic-bound}
 \frac{\exp(\tr M-n)}{\det M}
 \le\exp\left(\frac{n\eta^2}{2(1-\eta)}\right).
\end{align}
If in addition $\tr M = n$, then
\begin{align} \label{helly:determinant-bound}
 \frac{\int g}{\int h} = \alpha\det A
 \le(\det M)^{-1}\le(1-\eta^2)^{-n/2}.
\end{align}
\end{proposition}

\begin{proof}
For each $i$ with $w_i>0$, apply Lemma~\ref{helly:contact-inequality}
to $g\le F\le q_i$. Multiplying by $w_i$ and summing gives
\begin{align} \label{helly:summed-contact}
 \left(\sum_iw_it_i^2\right)\log\alpha
 +\left\langle\sum_iw_iu_i,a\right\rangle
 +\tr(AM)\le\tr M.
\end{align}
The barycenter equation cancels the translation term for every
competitor; the vertical moment fixes the coefficient of $\log\alpha$.
Hence
\begin{align} \label{helly:weighted-contact}
 \log\alpha+\tr(AM)\le\tr M.
\end{align}
Applying $\log x\le x-1$ to the eigenvalues of $M^{1/2}AM^{1/2}$ gives
$$
 \log\det A+\log\det M\le\tr(AM)-n.
$$
Combining this with \eqref{helly:weighted-contact} proves
\eqref{helly:determinant-general}.

If the eigenvalues $\mu_1,\ldots,\mu_n$ of $M$ belong to
$[1-\eta,1+\eta]$, then
$$
 \frac{\exp(\tr M-n)}{\det M}
 = \exp\left(\sum_{j = 1}^n(\mu_j-1-\log\mu_j)\right).
$$
Integrating $(t-1)/t$ between $1$ and $\mu_j$ yields
$$
 \mu_j-1-\log\mu_j\le\frac{(\mu_j-1)^2}{2(1-\eta)},
$$
which proves \eqref{helly:quadratic-bound}.

When $\tr M = n$, the exponential factor in
\eqref{helly:determinant-general} is one. The chord bound for the
concave logarithm gives
$$
 \log\mu_j\ge
 \frac{1+\eta-\mu_j}{2\eta}\log(1-\eta)
 +\frac{\mu_j-1+\eta}{2\eta}\log(1+\eta).
$$
Since $\sum_j\mu_j = n$, summation yields
$\log\det M\ge(n/2)\log(1-\eta^2)$ and proves
\eqref{helly:determinant-bound}.
\end{proof}

No integrability of $F$ or particular construction of the weights is
needed. The general estimate \eqref{helly:quadratic-bound} is already
quadratic in $\eta$; exact trace gives the sharper symmetric factor
in \eqref{helly:determinant-bound}.

For a positive definite $n\times n$ matrix $Q$, write
$$
 \mathcal D(Q) = \tr Q-\log\det Q-n.
$$
The factor in \eqref{helly:determinant-general} is $\exp(\mathcal D(M))$.
Before normalizing the moments, put
$$
 b = \sum_iw_i u_i,\quad v = \sum_iw_i t_i^2,\quad
 M = \sum_iw_i u_i u_i^T,
$$
for arbitrary nonnegative weights, and assume $v>0$ and $M\succ0$.
For every ellipsoidal
$g\le\min_{i:w_i>0}q_i$, \eqref{helly:summed-contact} and the
logarithmic determinant inequality applied to $M/v$ give
\begin{align} \label{helly:uncalibrated-contact}
 \log\frac{\int g}{\int h}
 \le\mathcal D(M/v)-\frac{\langle a,b\rangle}{v}.
\end{align}
Thus $b = 0$ removes the translation term, and division of the weights
by $v$ normalizes the vertical moment and replaces the covariance by $M/v$.

\begin{proposition}[Stability of the determinant bound] \label{helly:contact-defect}
Under the hypotheses of Proposition~\ref{helly:determinant}, let
$g = \alpha h(A^{-1}(\cdot-a))\le F$. Put
$$
 B = M^{1/2}AM^{1/2},\quad
 s_i = -t_i^2\log\alpha-\langle u_i,a\rangle
 -\langle u_i,(A-I_n)u_i\rangle
 \quad(i\text{ with }w_i>0).
$$
Then $s_i\ge0$ and
\begin{align} \label{helly:contact-defect-identity}
 \Delta := \mathcal D(M)-\log\frac{\int g}{\int h}
 = \sum_iw_i s_i+\mathcal D(B).
\end{align}
In particular, $\Delta\ge0$, $\sum_iw_i s_i\le\Delta$, and,
if $\Delta\le1/8$,
\begin{align} \label{helly:contact-defect-shape}
 \|M^{1/2}AM^{1/2}-I_n\|_{\mathrm F}^2\le4\Delta.
\end{align}
The sum in \eqref{helly:contact-defect-identity} is over the positive
weights.
\end{proposition}

\begin{proof}
Lemma~\ref{helly:contact-inequality} gives $s_i\ge0$.
The exact moments yield
$$
 \sum_iw_i s_i = \tr M-\log\alpha-\tr(AM),
$$
whereas
$$
 \mathcal D(B) = \tr(AM)-\log\det M-\log\det A-n.
$$
Adding proves \eqref{helly:contact-defect-identity}.
For $t>0$, the function $\psi(t) = t-1-\log t$ satisfies
$$
 \psi(t)\ge\frac{(t-1)^2}{2\max\{1,t\}}.
$$
This follows by integrating $(s-1)/s$ between $1$ and $t$.
Since $\psi(1/2)>1/8$ and $\psi(2)>1/8$, the condition
$\mathcal D(B)\le\Delta\le1/8$ puts every eigenvalue of $B$
in $[1/2,2]$. Applying the displayed inequality to those eigenvalues
proves \eqref{helly:contact-defect-shape}.
\end{proof}

The estimate controls the shape matrix $M^{1/2}AM^{1/2}$; the barycenter
equation has removed the translation parameter $a$.

\subsection{Sparse John contacts and selection of the original functions}

Functional John ellipsoids were introduced in \cite{hellyAGGMJV18}
by maximizing the integral of $\beta\mathbf1_E$ below a log-concave
function. We use the John $s$-functions of Ivanov-Nasz\'odi,
with $s = 1$, and the following specialization of their contact theorem
and integral-ratio estimate
\cite[Theorems~4.1, 5.1 and Corollary~6.1]{hellyIN22}.

\begin{theorem}[Ivanov-Nasz\'odi, $s = 1$] \label{helly:john-foundation}
Let $G:\R^n\to[0,\infty)$ be upper semicontinuous and log-concave,
with $0<\int G<\infty$. The maximum defining $J_1(G)$ exists and
is unique. If $J_1(G) = h$, there are finitely many points $(u_i,t_i)$
and weights $c_i>0$ such that
\begin{align} \label{helly:john-contacts}
 \begin{gathered}
 |u_i|^2+t_i^2 = 1,\quad t_i\ge0, \\
 \sum_i c_i u_i u_i^T = I_n,\quad
 \sum_i c_i u_i = 0,\quad \sum_i c_i t_i^2 = 1.
 \end{gathered}
\end{align}
Each $(u_i,t_i)$ belongs to the common boundary of $B_2^{n+1}$
and the closed lift $\mathcal L(G)$.
Moreover, an absolute constant $\Theta_1$ satisfies
\begin{align} \label{helly:classical-ratio}
 \int G\le(\Theta_1\sqrt n)^n\int J_1(G).
\end{align}
\end{theorem}

Reflection in the last coordinate preserves the displayed moments and
the contact property, so we may take $t_i\ge0$.

We may assume that the original functions are upper semicontinuous.
Indeed, let $D_i = \{f_i>0\}$ and let $\bar f_i$ be the upper
semicontinuous regularization of $f_i$. Since $\int f>0$, each convex
set $D_i$ has nonempty interior. A supporting affine function for
$\log f_i$ at an interior point bounds $f_i$ above by an exponential,
so $\bar f_i$ is finite; it is also log-concave. Continuity on
$\operatorname{int}D_i$ gives $\bar f_i = f_i$ there and outside
$\overline{D_i}$. The boundary of a full-dimensional convex set has
Lebesgue measure zero, so every subfamily minimum has the same integral
after regularization. Its ellipsoidal minorants are also unchanged:
if $g\le\bar f_i$ and $g(x)>0$, continuity of $g$ puts a neighborhood
of $x$ in $\overline{D_i}$. Thus $x\in\operatorname{int}D_i$, where
$g(x)\le f_i(x)$. The converse follows from $f_i\le\bar f_i$.
This justifies the reduction for both the integral and John-function
ratios, without adding a regularity hypothesis to the theorem.

Normalize the entire family by a common affine change of variable and
a positive rescaling so that $J_1(f) = h$. Explicitly, if
$J_1(f)(x) = \alpha h(T^{-1}(x-a))$, replace $f_j$ by
$\widetilde f_j(y) = f_j(a+Ty)/\alpha$.
The ellipsoidal class is invariant under this change, and every integral
is multiplied by $(\alpha\det T)^{-1}$. Thus selection cardinalities
and all ratios in the theorem are unchanged. In these coordinates,
$h\le f\le f_j$ for every $j$.

\begin{lemma}[Sparsification of the contact equations] \label{helly:sparse-contacts}
For the contact data in \eqref{helly:john-contacts} and $0<\eta<1$,
there are weights $w_i\ge0$, supported on at most the number in
\eqref{helly:count}, for which
\begin{align} \label{helly:exact-moments}
 \sum_iw_i u_i = 0,\quad
 \sum_iw_i t_i^2 = 1,\quad
 \sum_iw_i|u_i|^2 = n,
\end{align}
and
\begin{align} \label{helly:spectral-moment}
 M := \sum_iw_i u_i u_i^T,\quad
 (1-\eta)I_n\preceq M\preceq(1+\eta)I_n.
\end{align}
In particular, $\tr M = n$ and $\sum_iw_i = n+1$.
\end{lemma}

\begin{proof}
Apply Theorem~\ref{conic:psd} to $B_i = c_i u_i u_i^T$, whose sum
is $I_n$, with initial coefficients $\mathbf1$, and preserve the map
$$
 Lz = \left(\sum_i z_i c_i u_i,\quad
 \sum_i z_i c_i t_i^2,\quad
 \sum_i z_i c_i|u_i|^2\right).
$$
Its rank is at most $n+2$ and $L\mathbf1 = (0,1,n)$.
The resulting coefficients $s_i\ge0$ have support at most
$n+2+\lceil{C}n/\eta^2\rceil$ and satisfy $Ls = L\mathbf1$.
Set $w_i = s_ic_i$. This proves the two displays. The trace identity
is their third exact equation; summing the sphere identities gives
$\sum_iw_i = \sum_iw_i(|u_i|^2+t_i^2) = n+1$.
\end{proof}

The lift of $f$ is the finite intersection of the closed lifts of
the original functions. Consequently, each retained contact belongs
to the boundary of at least one $\mathcal L(f_j)$: otherwise it
would be interior to all the lifts and to their intersection.
Choose such an index $j(i)$ for each $i$ with $w_i>0$, and define
\begin{align} \label{helly:selection}
 J = \{j(i):w_i>0\},\quad F = f_J.
\end{align}
Then $h\le f\le F$, and each retained contact is a contact of the
lift of $F$. Since $\sum_iw_it_i^2 = 1$, the set $J$ is nonempty;
its cardinality is bounded by the number of retained contacts.

Apply Proposition~\ref{helly:determinant} with $q_i = f_{j(i)}$ on
the retained contacts. The equations \eqref{helly:exact-moments}
and \eqref{helly:spectral-moment} give all its hypotheses, including
$\tr M = n$. Thus every ellipsoidal minorant of the selected minimum
$F$ satisfies \eqref{helly:determinant-bound}.

\subsection{Integrability and proof of the Helly theorem}

\begin{proof}[Proof of Theorem~\ref{helly:functional}]
Normalize and select the contacts as above. Their assigned original
functions satisfy \eqref{helly:count}. For $R\ge1$, put
$F_R = F\mathbf1_{RB_2^n}$. This function is upper semicontinuous and
log-concave. Since $F$ is finite-valued and upper semicontinuous, it is
bounded on $RB_2^n$; moreover $F_R\ge h$. Thus
$0<\int F_R<\infty$, so the functional John theorem applies to $F_R$.
Because $J_1(F_R)\le F_R\le F$, Proposition~\ref{helly:determinant}
and \eqref{helly:classical-ratio} give
$$
 \int F_R
 \le(\Theta_1\sqrt n)^n\int J_1(F_R)
 \le(\Theta_1\sqrt n)^n(\det M)^{-1}\int h.
$$
Monotone convergence as $R\to\infty$ proves integrability of $F$
and the same bound for $\int F$. Since
$(\det M)^{-1}\le(1-\eta^2)^{-n/2}$ and $\int h\le\int f$,
this proves \eqref{helly:integral-ratio}.

The function $J_1(F)$ now exists. The function $h = J_1(f)$ is
admissible below $F$, and Proposition~\ref{helly:determinant} applies
to $J_1(F)$, proving \eqref{helly:john-ratio}. Undoing the common
normalization preserves all integral ratios. For $\eta = 1/2$,
\eqref{helly:count} gives $|J|\le(4C+1)n+2\le(4C+3)n$,
which gives \eqref{helly:main-bound} with the stated constants.
\end{proof}

For comparison, Lemma~\ref{matrix:elimination} applied directly to
$(u_iu_i^T,u_i,t_i^2)$ in \eqref{helly:john-contacts} preserves all the
contact equations with at most $n(n+1)/2+n+1$ contacts. Assigning them
to original functions gives $M = I_n$; the contact inequalities and
truncation argument then show that $J_1(f_J) = J_1(f)$.
Approximation reduces this quadratic selection count to a linear one.

\subsection{Volume and sharpness of the dimension dependence}

\begin{corollary}[Volume selection] \label{helly:volume}
Let $K_1,\ldots,K_N$ be convex subsets of $\R^n$ whose intersection
$K$ has finite positive volume. There is a nonempty $J\subseteq[N]$
with $|J|\le A_0n$ such that
$$
 \left|\bigcap_{j\in J}K_j\right|
 \le(C_0\sqrt n)^n|K|.
$$
More generally, the count \eqref{helly:count} and integral bound
\eqref{helly:integral-ratio} hold with integrals replaced by volumes.
\end{corollary}

\begin{proof}
Apply Theorem~\ref{helly:functional} to the log-concave functions
$f_i = \mathbf1_{K_i}$. Their minimum is
$\mathbf1_K$, of finite positive integral, and a selected minimum
is the indicator of the selected intersection.
\end{proof}

The order $\sqrt n$ is forced already by indicators. We use the
classical few-faces estimate in \cite[Theorem~1]{hellyBP}, attributed
there to Carl-Pajor and Gluskin: for $m\ge n$ and $v_j\in B_2^n$,
\begin{align} \label{helly:few-faces}
 \left|\{x:|\langle v_j,x\rangle|\le1, 1\le j\le m\}\right|^{1/n}
 \ge\frac{c}{\sqrt{1+\log(m/n)}}
\end{align}
with an absolute $c>0$; the volume is infinite if the vectors do not span.

\begin{proposition}[Sharp order when at most $An$ sets are selected]
 \label{helly:sharpness}
For every $A\ge2$ and every $n\ge1$, there is a finite family of
compact convex sets with intersection of positive volume such that
every nonempty selection of at most $An$ members has volume ratio
to the power $1/n$ at least
$$
 c\sqrt{\frac{n}{1+\log A}},
$$
where $c>0$ is absolute. The same lower bound holds for the ratio
of integrals of the corresponding log-concave indicator functions.
\end{proposition}

\begin{proof}
Choose finitely many unit vectors $v_i$ whose tangent halfspaces
$H_i = \{x:\langle v_i,x\rangle\le1\}$ have intersection $K$ with
$B_2^n\subseteq K\subseteq2B_2^n$; a finite cover of the unit sphere
provides such a family. For a selection $S$ of size $m\le An$, the
polyhedron $P_S = \bigcap_{i\in S}H_i$ contains the symmetric slab
intersection in \eqref{helly:few-faces}. If $m<n$, it has infinite
volume because it contains $B_2^n$ and a nontrivial common lineality
space. Thus in all cases
$$
 |P_S|^{1/n}\ge\frac{c}{\sqrt{1+\log A}},\quad
 |K|^{1/n}\le2|B_2^n|^{1/n}\le\frac{C}{\sqrt n}.
$$
Replace every $H_i$ by $H_i\cap RB_2^n$. For each of the finitely
many selections under consideration, monotone convergence gives
$|P_S\cap RB_2^n|\uparrow|P_S|$ as $R\to\infty$.
Choose a common $R\ge2$ so that every selected volume retains the
preceding lower bound on its $n$th root with an additional factor
$1/2$. The full intersection is still $K$, and the resulting sets
are compact. Dividing the two bounds proves the assertion.
\end{proof}

\section{Simultaneous inverse moments and sparse representations} \label{framework:section}

We now vary both the vectors in a covariance representation and their
coefficients. The first variation will control inverse moments, while
the second will bound the coefficients of the same representation.
We first work with an inverse-trace objective over bases. The general
covariance family then identifies the corresponding variations for
$0\le p\le1$, including the determinant endpoint used in the known
cube upper bound.

All spaces are finite-dimensional real inner-product spaces. Singular
values are decreasing and extended by zero.

\subsection{Inverse-trace minimization over bases}
 \label{framework:sec-covariance}

Let $K = \conv\{y_1,\ldots,y_N\} = -K$ span $\R^n$, with $n\ge1$.
In particular $N\ge2$.
An inscribed basis is a matrix $X\in\operatorname{GL}(n)$ with
$Xe_i\in K$ for all $i$. Write
$$
 \beta(X) = \sum_{i = 1}^n\|X^{-T}e_i\|_2.
$$
Write $\Delta_N = \{q \ge 0:\sum_jq_j = 1\}$. For $q\in\Delta_N$ set
$$
 h_j(q) = q_j+N^{-1},\quad
 H(q) = \sum_jh_j(q)\log\frac{2}{h_j(q)},\quad
 \lambda = \frac1{10\log N}.
$$
For coefficient columns $p_i\in\Delta_N$ and positive weights $W_i$, put
$$
 x_i = \sum_jp_{ji}y_j,\quad X = [x_1,\ldots,x_n],\quad
 W = \operatorname{diag}(W_1,\ldots,W_n),\quad
 h_{ji} = p_{ji}+N^{-1}.
$$
We minimize
\begin{align} \label{cube:objective}
 \Psi(p,W) = \frac12\sum_{i = 1}^n
 \left(\frac{\|X^{-T}e_i\|_2^2}{W_i}
 +W_i(1+2\lambda H(p_i))\right),
\end{align}
with value $+\infty$ when $X$ is singular.

\begin{theorem}[Inverse-trace minimization over bases]
 \label{framework:square}
The objective \eqref{cube:objective} attains a minimum, and its minimum
is at most $2\inf_{X_0}\beta(X_0)$, where the infimum is over all
inscribed bases. At every minimum, put
$$
 v_i = X^{-T}e_i,\quad w_i = \|v_i\|_2,\quad
 Y = XW^{1/2},\quad M = (YY^T)^{-1}.
$$
Then
\begin{align} \label{cube:budgets}
 \frac{w_i}{2} \le W_i \le w_i,\quad
 \beta(X) \le \operatorname{tr}M = \Psi(p,W),\quad
 \operatorname{tr}W \le \operatorname{tr}M,
\end{align}
and for every $y\in K$,
\begin{align} \label{cube:cubic}
 |y^TMv_i| \le 4w_i,\quad y^TM^3y \le 16\operatorname{tr}M.
\end{align}
In particular, if some inscribed basis $X_0$ satisfies $\beta(X_0)\le t/2$,
then $\tr M\le t$ and $y^TM^3y\le16t$ for $y\in K$.
For a nonempty set $I\subseteq\{1,\ldots,N\}$ define
$$
 k_I = |\{(j,i):j\in I, p_{ji}>0\}|,\quad
 e_I = \dim\operatorname{span}\{y_j:j\in I\},\quad
 \rho_I = \sum_{j\in I}y_j^TMy_j.
$$
If $\mu_1 \ge \cdots \ge \mu_n>0$ are the eigenvalues of $M$, then
\begin{align}
 k_I& \le 2n+\frac2\lambda
 \max_{j\in I}(y_j^TMy_j)\operatorname{tr}M,
 \label{framework:square-col} \\
 k_I& \le n+e_I^2+
 2\left(\frac{2\rho_I}{\lambda}
 \sum_{\ell = 1}^n\ell\mu_\ell\right)^{1/2}\notag\\
 & \le n+e_I^2+2\sqrt{\frac{2\rho_I}{\lambda}}\operatorname{tr}\sqrt M.
 \label{framework:square-collect}
\end{align}
The bounds hold simultaneously for every row set $I$. For the empty
set the counts, the sums and the maximum are taken to be zero.
\end{theorem}

The basis has $n$ vectors; sparsity concerns their coefficients in the
given hull. The separate-column bound \eqref{framework:square-col}
uses the largest row energy. The collective bound
\eqref{framework:square-collect} instead uses the row span, the total
row energy, and the inverse spectrum. We prove attainment and the
moment bounds first, followed by both coefficient estimates.

\subsection{First variation and the coefficient penalty}

The shift in $H$ permits first-order tests toward generators whose
current coefficient is zero. For the unshifted entropy
$-\sum_jq_j\log q_j$, introducing a missing coefficient gives an
infinite one-sided derivative. It would therefore give no finite
inverse-moment bound from that test. Here $h_j \ge N^{-1}$ and
$\sum_jh_j = 2$. Since $(h_j/2)_j$ is a probability vector, $H$ is
twice its entropy. The penalty is smooth on a neighborhood of the
simplex and satisfies
$$
 0 \le H(q) \le 2\log N,\quad
 D^2H(q)[z,z] = -\sum_j\frac{z_j^2}{h_j(q)},\quad
 DH(q)[q'-q] \le \log(2N).
$$
For the last inequality, the constant terms in the derivative cancel,
and the difference of the two convex averages of $\log h_j$ is at
most $\log(N+1)$. Thus $\lambda\log N$ controls the error in the
first variation, while $\lambda$ supplies a definite negative second
variation after normalizing the coefficient coordinates. Strict
concavity and permutation invariance also imply that all simplex
vertices have equal entropy, strictly less than that of any nonvertex.

\begin{proof}[Attainment and moment bounds in Theorem~\ref{framework:square}]
For fixed invertible $X$, minimization over the weights gives
$$
 W_i = \frac{w_i}{\sqrt{1+2\lambda H(p_i)}},\quad
 \inf_W\Psi(p,W) = \sum_iw_i\sqrt{1+2\lambda H(p_i)}.
$$
This value lies between $\beta(X)$ and $2\beta(X)$ and diverges at
singular representations, since $\beta(X) \ge \|X^{-1}\|_{\mathrm{op}}$.
The spanning assumption ensures the existence of an inscribed basis.
The product of coefficient simplices is compact. Hence the minimum
is attained, and comparison with any inscribed basis $X_0$
bounds its value by $2\beta(X_0)$. The weight equations give
$$
 \frac{w_i^2}{W_i} = W_i(1+2\lambda H(p_i)),\quad
 M = X^{-T}W^{-1}X^{-1}.
$$
Summing proves \eqref{cube:budgets}.

For the first variation fix $W$ and move $p_i$ toward $q'\in\Delta_N$
representing any $y\in K$. Write
$c_i = x_i^TM^2x_i = w_i^2/W_i^2 = 1+2\lambda H(p_i)$.
The derivative of the inverse-trace contribution to $\Psi$ is
$-W_i(y-x_i)^TM^2x_i$. Hence
$$
 y^TM^2x_i \le c_i+\lambda DH(p_i)[q'-p_i]
 \le c_i+\lambda\log(2N).
$$
Applying the same test to $-y\in K$ bounds the absolute value.
Since $Mx_i = v_i/W_i$ and $N\ge2$, it follows that
$$
 |y^TMv_i| \le W_i(c_i+\lambda\log(2N)) \le 4w_i.
$$
Finally,
$$
 \begin{aligned}
 y^TM^3y
 & = \sum_iW_i(y^TM^2x_i)^2
 = \sum_iW_i^{-1}(y^TMv_i)^2 \\
 & \le 16\sum_iw_i^2/W_i = 16\operatorname{tr}M.
 \end{aligned}
$$
This proves \eqref{cube:cubic}. The third inverse moment comes from
squaring and summing the linear tests supplied by the coefficient
variation; it is not an additional constraint in the objective.
\end{proof}

\begin{remark}[Relative covariance approximation and inverse moments]
 \label{framework:inverse-moment-example}
A relative approximation of a covariance matrix alone need not control
higher inverse moments in a fixed direction. For $\tau>0$ and
$0<\varepsilon<1$, put
$$
 A_0 = \begin{pmatrix}\tau&0 \\0&1\end{pmatrix},\quad
 A_\varepsilon = 
 \begin{pmatrix}\tau&\varepsilon\sqrt\tau \\
 \varepsilon\sqrt\tau&1\end{pmatrix}.
$$
Then $(1-\varepsilon)A_0\preceq A_\varepsilon
\preceq(1+\varepsilon)A_0$, whereas direct inversion gives
$$
 e_2^TA_0^{-3}e_2 = 1,\quad
 e_2^TA_\varepsilon^{-3}e_2
 = \frac{1+2\varepsilon^2/\tau+\varepsilon^2/\tau^2}
 {(1-\varepsilon^2)^3}.
$$
For fixed $\varepsilon$, the latter tends to infinity as $\tau\downarrow0$.
The moment estimate in Theorem~\ref{framework:square} is obtained from
the optimizing representation itself, together with its coefficient bound.
\end{remark}

\subsection{Nilpotent variations and coefficient support}
 \label{framework:sec-operators}

A variation of the coefficients changes $X$ to $(I+sQ)X$. On a linear
space of variations for which $Q^2 = 0$, the inverse is exactly affine
in $s$. This removes the higher inverse terms from the second
variation and leaves a positive tensor norm to compare with the
negative curvature of $H$.

\begin{lemma}[Second variation] \label{cube:second}
At a minimum in Theorem~\ref{framework:square}, let $z_{ji} = 0$
whenever $p_{ji} = 0$, and suppose
$$
 \sum_jz_{ji} = 0\quad(1 \le i \le n),\quad
 Q = \sum_{j,i}z_{ji}y_jv_i^T,\quad Q^2 = 0.
$$
Then
\begin{align} \label{cube:second-equation}
 0 \le \left.\frac{d^2}{ds^2}\right|_{s = 0}\Psi(p+sz,W)
 = \|M^{1/2}Q\|_{\mathrm F}^2
 -\lambda\sum_{j,i}\frac{W_i z_{ji}^2}{h_{ji}}.
\end{align}
\end{lemma}

\begin{proof}
Both signs of small $s$ are feasible and $X(s) = (I+sQ)X$.
Since $Q^2 = 0$, $X(s)^{-1} = X^{-1}(I-sQ)$ exactly. Thus the inverse
term has second derivative
$\sum_iW_i^{-1}\|Q^Tv_i\|_2^2 = \|M^{1/2}Q\|_{\mathrm F}^2$;
the entropy contributes the displayed negative term.
\end{proof}

The nilpotence condition can be enforced on a linear space of
variations with controlled codimension. For variations in rows $I$,
put $E = \operatorname{span}\{y_j:j\in I\}$. Their matrices $Q$
already satisfy $\operatorname{range}Q\subseteq E$. Imposing
$Q|_E = 0$ therefore gives, relative to $E\oplus E^\perp$,
$$
 Q = \begin{pmatrix}0&*\\0&0\end{pmatrix},\quad
 Q^2 = 0,\quad (I+sQ)^{-1} = I-sQ.
$$
Since $Q|_E$ takes values in $E$, its vanishing costs at most
$e_I^2$ linear equations. Together with the $n$ column-sum equations,
these sufficient conditions define a variation space of codimension
at most $n+e_I^2$.

For a single column $i$, one can use fewer restrictions:
$Q = (\sum_jz_{ji}y_j)v_i^T$ has rank at most one, so the scalar
equation $v_i^T\sum_jz_{ji}y_j = 0$ ensures $Q^2 = 0$. Including the
column-sum equation costs at most two dimensions. These two choices
of variation space give the two coefficient bounds.

For variations supported on a set $\Omega$ of active coefficients,
normalize the coordinates by
$$
 z_{ji} = \sqrt{h_{ji}/W_i}\xi_{ji}\quad((j,i)\in\Omega),\quad
 z_{ji} = 0\quad((j,i)\notin\Omega).
$$
With $a_j = M^{1/2}y_j$ and $b_i = v_i/\sqrt{W_i}$, the positive
curvature operator is
$$
 T\xi = M^{1/2}Q
 = \sum_{(j,i)\in\Omega}\sqrt{h_{ji}}\xi_{ji}a_jb_i^T.
$$
On every linear space satisfying the column-sum equations and
$Q^2 = 0$, Lemma~\ref{cube:second} gives
$\|T\xi\|_{\mathrm F}^2 \ge \lambda\|\xi\|^2$.
The next lemma approximates this tensor operator by allocating a
separate rank to each row. Its abstract form will also be used for
the general covariance family.

\begin{lemma}[Row-dependent spectral approximation]
 \label{framework:tensor}
Let $a_1,\ldots,a_m$ and $b_1,\ldots,b_n$ belong to Euclidean spaces.
Let $\Omega\subseteq[m]\times[n]$, and let $h_{ji} \ge 0$ on $\Omega$,
extended by zero. Define
$$
 T e_{ji} = \sqrt{h_{ji}} a_j\otimes b_i,\quad
 H = \max_i\sum_jh_{ji},\quad S = \sum_j|a_j|^2,\quad
 B = [b_1,\ldots,b_n].
$$
The domain is $\R^\Omega$ with its coordinate Euclidean structure
and the tensor product carries the Hilbert norm.
Then
\begin{align}
 (T^{*}T)_{(j,i),(k,l)}
 & = \sqrt{h_{ji}h_{kl}}\langle a_j,a_k\rangle\langle b_i,b_l\rangle,
 \label{framework:eq-tensor-gram} \\
 \|T\|_{\HS}^2& = \sum_{j,i}h_{ji}|a_j|^2|b_i|^2.
 \label{framework:eq-tensor-HS}
\end{align}
Let $f_\ell$ be left singular vectors of $B$, and let $P_k$ denote
projection onto the first $k$ such vectors. For any integers
$0 \le m_j \le \rank B$, the map
$T_0e_{ji} = \sqrt{h_{ji}} a_j\otimes P_{m_j}b_i$ satisfies
\begin{align}
 \label{framework:eq-tensor-tail-exact}
 \rank T_0 \le \sum_jm_j,\quad
 \|T-T_0\|_{\HS}^2
 = \sum_{j,i}h_{ji}|a_j|^2
 \sum_{\ell>m_j}|\langle f_\ell,b_i\rangle|^2.
\end{align}
In particular for each $\alpha>0$ there is $T_\alpha$ with
\begin{align}
 \label{framework:eq-tensor-tail}
 \rank T_\alpha \le \alpha S,\quad
 \|T-T_\alpha\|_{\HS}^2
 \le \frac H\alpha\sum_{\ell \ge 1}\ell s_\ell(B)^2.
\end{align}
\end{lemma}

\begin{proof}
The Gram and Hilbert-Schmidt formulas follow by evaluation on
coordinate unit vectors. The stated range of $T_0$ is contained
in the span of the $a_j\otimes f_\ell$ for $\ell \le m_j$.
Orthogonality of the $f_\ell$ gives the exact error.
Choose $m_j = \min\{\rank B,\lfloor\alpha|a_j|^2\rfloor\}$.
Then $\sum m_j \le \alpha S$. If $\ell>m_j$ and
$\ell \le \rank B$, then $|a_j|^2<\ell/\alpha$. Thus
$$
 \|T-T_\alpha\|_{\HS}^2
 \le \frac1\alpha\sum_\ell\ell
 \sum_i|\langle f_\ell,b_i\rangle|^2\sum_jh_{ji}
 \le \frac H\alpha\sum_\ell\ell s_\ell(B)^2 .
$$
This proves the lemma.
\end{proof}

\begin{corollary}[Bounds for the number of coefficients]
 \label{framework:cor-tensor}
In the notation of Lemma~\ref{framework:tensor}, suppose
$\|N\xi\| \le \|T\xi\|$ on a space of codimension at most $b$.
If $N^{*}N\succeq\kappa I$ with $\kappa>0$, then
\begin{align}
 |\Omega|& \le b+\frac1\kappa\sum_{j,i}h_{ji}|a_j|^2|b_i|^2,
 \label{framework:eq-tensor-dim-hs} \\
 |\Omega|& \le b+2\left(\frac{HS}{\kappa}\sum_{\ell\ge1}\ell s_\ell(B)^2\right)^{1/2}.
 \label{framework:eq-tensor-dim-lorentz}
\end{align}
More precisely, for arbitrary $m_j$ as in
\eqref{framework:eq-tensor-tail-exact}, one may bound $|\Omega|$ by $b$ plus
$\sum_jm_j+\kappa^{-1}\|T-T_0\|_{\HS}^2$. Without the uniform
curvature assumption, \eqref{framework:eq-sing-master} gives the corresponding
tail bounds for the full operator $N$.
\end{corollary}

\begin{proof}
Apply \eqref{framework:eq-tail-master} with the indicated approximations.
Optimizing $\alpha S+\frac{H}{\alpha\kappa}\sum_{\ell\ge1}\ell s_\ell(B)^2$
over $\alpha>0$ gives \eqref{framework:eq-tensor-dim-lorentz}; zero factors
are handled by taking the appropriate limit.
\end{proof}

\begin{proof}[Coefficient bounds in Theorem~\ref{framework:square}]
Fix a row set $I$. For a single column $i$, use its active
coefficients in these rows and the variation space of codimension
at most two described above.
In the normalized coordinates above, the squared Hilbert-Schmidt norm
of the curvature operator is at most
$$
 2\max_{j\in I}(y_j^TMy_j)\frac{w_i^2}{W_i},
$$
because the sum of the $h_{ji}$ over any subset of a column is at
most $2$. Apply \eqref{framework:eq-tensor-dim-hs} with
$\kappa = \lambda$ and sum over the columns. The identity
$\sum_iw_i^2/W_i = \operatorname{tr}M$ gives
\eqref{framework:square-col}.

For collective variations in the rows $I$, the space described above
has codimension at most $n+e_I^2$.
Use Corollary~\ref{framework:cor-tensor} with $H \le 2$,
$S = \rho_I$ and $\kappa = \lambda$. Since
$$
 [b_i][b_i]^T = \sum_iW_i^{-1}v_iv_i^T = M,
$$
its collective bound is the first inequality in
\eqref{framework:square-collect}. Decreasing order of the eigenvalues
gives
\begin{align} \label{framework:trace-majorization}
 \left(\operatorname{tr}\sqrt M\right)^2
 = \sum_\ell\mu_\ell+2\sum_{i<j}\sqrt{\mu_i\mu_j}
 \ge \sum_\ell(2\ell-1)\mu_\ell
 \ge \sum_\ell\ell\mu_\ell.
\end{align}
This proves the remaining inequality. The argument applies to every
row set at each fixed minimum, so all the estimates are simultaneous.
\end{proof}

The row-dependent cutoffs in Lemma~\ref{framework:tensor} are important
when the inverse spectrum has a long tail. For instance, if
$\mu_\ell\asymp a/\ell^2$, $1 \le \ell \le n$, then
$$
 \sum_\ell\ell\mu_\ell\asymp a\log n,\quad
 \left(\tr\sqrt M\right)^2\asymp a(\log n)^2
 \quad(n \ge 2).
$$
The inverse singular-value estimate in Theorem~\ref{cube:completion}
gives a bound of this form in the tail. Using the weighted sum directly
in \eqref{framework:square-collect} replaces a factor $(\log n)^2$
by $\log n$ in the spectral estimate under the square root.

\subsection{The general covariance family}

The inverse-trace calculation extends to a spectral family in which
the exponent determines the inverse-moment estimates and the second
variation records the curvature available for coefficient bounds.
Allowing a regularizing covariance and a fixed total weight also
places the determinant maximization in the cube upper bound in this
family. The theorem permits arbitrary numbers of vectors.

Here $K = \conv\{y_j:1\le j\le N\}$ need not be symmetric,
$N\ge1$ is allowed, and
$$
 \mathcal C_K = \conv\{yy^T:y\in K\}.
$$
The shifted entropy $H$ is as above, with $\lambda\ge0$ now a free
parameter. For $0\le p\le1$ and $A\succ0$, define
$$
 \Phi_p(A) = 
 \begin{cases}
 (\tr A^{-p}-n)/p,&p>0, \\
 -\log\det A,&p = 0 .
 \end{cases}
$$
For each fixed $A$, $\Phi_p(A)\to-\log\det A$ as $p\downarrow0$.
Given $m \ge 1$ and $\tau,\lambda,a \ge 0$, take
$$
 q_i\in\Delta_N,\quad x_i = \sum_jq_{ji}y_j,\quad
 w_i \ge 0,\quad C\in\mathcal C_K,\quad
 S = \sum_iw_ix_ix_i^T,\quad A = S+\tau C.
$$
Minimize
\begin{align}
 \label{framework:eq-cov-objective}
 J(q,w,C) = \frac12\Phi_p(A)+\frac a2\sum_iw_i+
 \lambda\sum_iw_iH(q_i)
\end{align}
with value $+\infty$ when $A$ is singular. The parameter $\tau$
measures covariance regularization, $a$ controls the total weight,
and $\lambda$ controls the coefficient penalty. We consider either
fixed positive mass $\sum_iw_i$, or free weights with $a>0$, and
assume that a feasible positive definite $A$ exists.
Write $J_{p,\tau,a,\lambda}$ for this objective.
For $m = n$, $p = 1$, $\tau = 0$, $a = 1$ and free weights, it recovers the
inverse-trace objective:
\begin{align} \label{cube:inverse-endpoint}
 J_{1,0,1,\lambda}+\frac n2 = \Psi.
\end{align}

The linear restrictions used for bases extend to this setting.
Let $X$ contain the $r$ positive-weight columns and set
$W = \operatorname{diag}(w_i)$ and $M = A^{-1}$.
For $De_i = \sum_jz_{ji}y_j$, the corresponding matrix is
$$
 R = M^{1/2}D W X^TM^{1/2}.
$$
For variations in rows $I$, its range lies in
$E = M^{1/2}\operatorname{span}\{y_j:j\in I\}$. As before,
$R|_E = 0$ ensures $R^2 = 0$ at a cost of at most $(\dim E)^2$
linear equations. For a single column, one orthogonality equation
suffices. If $r=n$ and $\tau=0$, then $X$ is invertible and, with
$Q=DX^{-1}$, one has $R=M^{1/2}QM^{-1/2}$.
The second variation now has two additional nonpositive terms:
one comes from the positive semidefinite defect $I-B^TB$, where
$B=M^{1/2}XW^{1/2}$, and the other from the concavity of the spectral
function of the inverse. Keeping both terms explains the behavior
at the two endpoints.

\begin{theorem}[Variational covariance estimates]
 \label{framework:covariance}
The objective \eqref{framework:eq-cov-objective} attains a finite minimum.
At any finite local minimizer put
$$
 M = A^{-1},\quad G = M^{p+1},\quad
 c_i = x_i^TGx_i,\quad
 \Lambda_p = \max_{y\in K}y^TGy.
$$
Retain the $r$ positive-weight columns in $X = [x_i]$ and
$W = \diag(w_i)$, so $S = XWX^T$. Define
$$
 B_p = M^{(p+1)/2}S M^{(p+1)/2},\quad
 b_i = \sqrt{w_i}M^{(p+1)/2}x_i,\quad
 \beta_p = \tr B_p.
$$
The following assertions hold.

\emph{(a) First-order conditions.} There is $\eta\in\R$, equal to zero
in the free-weight regime, such that
$$
 c_i = a+2\lambda H(q_i)+\eta
$$
for positive weights, with $c_i \le a+2\lambda H(q_i)+\eta$ at
zero weights. If $\tau>0$, $\tr(GC) = \Lambda_p$. Hence
\begin{align}
 \label{framework:eq-cov-budget}
 B_p = \sum_i b_ib_i^T = M^p-\tau M^{(p+1)/2}CM^{(p+1)/2},
 \quad
 \beta_p = \sum_iw_ic_i = \tr M^p-\tau\Lambda_p.
\end{align}
In particular $0\preceq B_p\preceq M^p$.

\emph{(b) Inverse-moment estimates.} For $y\in K$ and positive weights,
$y^TGx_i \le c_i+\lambda \log(2N)$. If $K = -K$, then
\begin{align}
 \label{framework:eq-cov-moment}
 |y^TGx_i| \le c_i+\lambda \log(2N),\quad
 y^TGSGy \le \sum_iw_i(c_i+\lambda \log(2N))^2.
\end{align}
For $\tau = 0$ the left quadratic form in the second inequality is
$y^TM^{2p+1}y$.

\emph{(c) Second variation.} Let $\Omega$ be any subset of the
positive coefficients in positive-weight columns and use the
normalized coordinates
$$
 z_{ji}(\xi) = \sqrt{h_{ji}/w_i} \xi_{ji}
 \quad((j,i)\in\Omega),\quad z_{ji} = 0\text{ otherwise},
$$
where $h_{ji} = q_{ji}+N^{-1}$. Write
$D_\xi e_i = \sum_jz_{ji}y_j$ and set
$$
 B = M^{1/2}XW^{1/2},\quad Z_\xi = M^{1/2}D_\xi W^{1/2},
 \quad R_\xi = Z_\xi B^T,\quad H_\xi = R_\xi+R_\xi^T,
 \quad M'_\xi = -M^{1/2}H_\xi M^{1/2}.
$$
Let $V\subseteq\R^\Omega$ be any linear space on which
$\sum_jz_{ji} = 0$ for every $i$ and $R_\xi^2 = 0$. Define
$$
 T\xi = R_\xi M^{p/2}
 = \sum_{(j,i)\in\Omega}\sqrt{h_{ji}}\xi_{ji}
 (M^{1/2}y_j)b_i^T.
$$
Let $\mathcal G_p(M) = \tr M^p/p$ for $p>0$ and
$\mathcal G_0(M) = \log\det M$. Represent the following nonnegative quadratic forms by linear maps
$D_1,D_2$ on the entire coefficient space:
$$
 \|D_1\xi\|^2
 = \|M^{p/2}Z_\xi(I-B^TB)^{1/2}\|_{\F}^2,\quad
 \|D_2\xi\|^2
 = -\frac12D^2\mathcal G_p(M)[M'_\xi,M'_\xi].
$$
Then
\begin{align}
 \label{framework:eq-cov-signed}
 \|T\xi\|_{\F}^2 \ge 
 \lambda\|\xi\|^2+\|D_1\xi\|^2+\|D_2\xi\|^2
 \quad(\xi\in V).
\end{align}
In particular Lemma~\ref{framework:master} applies to
$N = (\sqrt\lambda I,D_1,D_2)$ with $b = \codim V$.

\emph{(d) Support and approximation.} If $\lambda>0$,
$$
 (|\Omega|-b-k)_+\lambda
 \le \|T-T_0\|_{\HS}^2
 \quad\text{ for every rank-$k$ approximation }T_0.
$$
The Hilbert-Schmidt estimate is
\begin{align}
 \label{framework:eq-cov-HS}
 \|T\|_{\HS}^2
 = \sum_{(j,i)\in\Omega} w_ic_i h_{ji} y_j^TMy_j
 \le 2 R_0\beta_p
\end{align}
when $y_j^TMy_j \le R_0$ for every row in $\Omega$.
For $\lambda>0$ and any nonempty row set $I$, let
$e_I = \dim\spanop\{y_j:j\in I\}$ and
$\rho_I = \sum_{j\in I}y_j^TMy_j$. If $k_I$ counts the active
positions in these rows and positive-weight columns, then
\begin{align}
 k_I& \le 2r+\frac{2}{\lambda}
 \max_{j\in I}(y_j^TMy_j) \beta_p,
 \label{framework:eq-cov-col} \\
 k_I& \le r+e_I^2+
 2\left(\frac{2\rho_I}{\lambda}\sum_{\ell = 1}^n\ell\lambda_\ell(B_p)\right)^{1/2},
 \label{framework:eq-cov-collect}
\end{align}
where $\lambda_\ell(B_p)$ are the eigenvalues of $B_p$ in decreasing
order. One may retain the full Gram matrix from \eqref{framework:eq-tensor-gram}
and the individual row cutoffs from \eqref{framework:eq-tensor-tail-exact}
instead of these coarser estimates.

\emph{(e) Replacement of one column.} Suppose additionally the optimizer is
optimal against every replacement of a single coefficient column
with the other variables fixed. For $i$ of positive weight,
$q'\in\Delta_N$, $y = \sum_jq'_jy_j$ and
$A_y = A-w_ix_ix_i^T+w_iyy^T\succ0$, one has
\begin{align}
 \label{framework:eq-cov-exchange}
 \frac{\Phi_p(A)-\Phi_p(A_y)}{w_i}
 \le 2\lambda(H(q')-H(q_i)).
\end{align}
The removed-column matrix $A-w_ix_ix_i^T$ need not be invertible.
\end{theorem}

\begin{proof}[Proof of Theorem~\ref{framework:covariance}]
\emph{Attainment.}
For attainment, the variables $q_i$ and $C$ range over compact
sets, and
$\tr A \le R_{*}^2(\sum_iw_i+\tau)$ where $R_{*} = \max_j|y_j|$.
The fixed-mass case follows by compactness, since $\Phi_p(A)$
diverges at the singular boundary when $\tr A$ is bounded.
For free weights, write $s = \sum_iw_i$. When $p>0$,
$J \ge -n/(2p)+as/2$. When $p = 0$, the arithmetic-geometric
mean bound on $\det A$ gives
$$
 J \ge \frac a2s-\frac n2\log(R_{*}^2(s+\tau)/n).
$$
Each lower bound tends to infinity as $s\to\infty$.
Consequently every nonempty finite sublevel set is compact,
and the feasible positive definite representation ensures a
finite minimum.

\emph{First variation: covariance and inverse moments.}
The differential of $\Phi_p$ is
$D\Phi_p(A)[D] = -\tr(M^{p+1}D)$. This follows by the differential
of the trace of a spectral function and the log determinant formula.
Varying positive weights independently, or in pairs when mass is
fixed, gives (a); moving mass from a positive to a zero weight gives
the stated inequality. A variation of $C$ toward $yy^T$ gives
$\tr(GC) \ge y^TGy$ if $\tau>0$. The reverse inequality follows by
convex combination. The identities \eqref{framework:eq-cov-budget} follow
from $S = A-\tau C$. The maximum defining $\Lambda_p$ may be taken over
the generators because $G\succeq0$.

For a positive-weight column, move $q_i$ towards any $q'$.
The derivative of the covariance contribution is
$-w_i(y-x_i)^TGx_i$, so first-order optimality gives
$(y-x_i)^TGx_i \le \lambda DH(q_i)[q'-q_i] \le \lambda \log(2N)$.
Use $-y\in K$ for the other sign, then square and sum with weights
$w_i$ to prove \eqref{framework:eq-cov-moment}.

If $r=0$, then $S=B_p=0$ and $\beta_p=0$, so the remaining
coefficient-variation, support and replacement assertions are
trivial or vacuous. We may therefore assume $r\ge1$.

\emph{Second variation on the chosen linear space.}
We now compute the second variation exactly. For fixed $w,C$ and
$X(s) = X+sD_\xi$, write $Z = Z_\xi$, $R = R_\xi$ and $H = H_\xi$.
The inverse $M(s) = A(s)^{-1}$ satisfies
$$
 A'(0) = M^{-1/2}HM^{-1/2},\quad
 A''(0) = 2M^{-1/2}ZZ^TM^{-1/2},
$$
$$
 M'(0) = -M^{1/2}HM^{1/2},\quad
 M''(0) = 2M^{1/2}(H^2-ZZ^T)M^{1/2}.
$$
Because $BB^T = M^{1/2}SM^{1/2}\preceq I$, also $B^TB\preceq I$.
On $R^2 = 0$,
$$
 H^2-ZZ^T = R^TR-Z(I-B^TB)Z^T.
$$
The spectral function $\mathcal G_p$ is concave for $0\le p\le1$.
Indeed, the trace of a concave scalar function of a positive definite
matrix is concave: in an eigenbasis of a convex combination of two
matrices, apply scalar Jensen's inequality to the spectral resolutions
of the two matrices and sum. Here the scalar function is $x^p/p$ for
$p>0$ and $\log x$ for $p = 0$.

Since $D\mathcal G_p(M)[U] = \tr(M^{p-1}U)$, the chain rule gives
\begin{align}
 \label{framework:eq-cov-hess-exact}
 \frac12(\Phi_p(A(s)))''\big|_0
 = \|RM^{p/2}\|_{\F}^2
 -\|M^{p/2}Z(I-B^TB)^{1/2}\|_{\F}^2
 +\frac12D^2\mathcal G_p(M)[M'_\xi,M'_\xi].
\end{align}
The entropy second derivative is
$-\lambda\sum_{j,i}w_i z_{ji}^2/h_{ji}
 = -\lambda\|\xi\|^2$. The remaining term in $J$ is constant.
On $V$ the paths are feasible for small steps of either sign.
Local minimality makes the second derivative nonnegative on $V$,
which proves \eqref{framework:eq-cov-signed}.
Nonnegative quadratic forms admit linear square-root
factorizations, which justifies the definition of $D_2$.

\emph{From curvature to support bounds.}
Lemma~\ref{framework:master} applied to \eqref{framework:eq-cov-signed}
gives the rank-approximation assertion in (d). The tensor identity for
$T$ follows from the definitions, and Lemma~\ref{framework:tensor}
gives \eqref{framework:eq-cov-HS}, since the sum of the $h_{ji}$
over any subset of a column is at most $2$.
The single-column and collective variation spaces have codimensions
at most $2$ per column and $r+e_I^2$, respectively, by the construction
preceding the theorem. Apply \eqref{framework:eq-tensor-dim-hs}
separately to the columns and sum, using $\sum_i|b_i|^2 = \beta_p$,
to obtain \eqref{framework:eq-cov-col}. Apply
\eqref{framework:eq-tensor-dim-lorentz} with $H \le 2$ and
$S = \rho_I$, noting that $[b_i][b_i]^T = B_p$, to obtain
\eqref{framework:eq-cov-collect}.

\emph{Finite replacement.}
Finally, compare \eqref{framework:eq-cov-objective} at $q_i$ and $q'$.
All terms except the covariance and the one entropy term cancel.
This proves \eqref{framework:eq-cov-exchange}. Singular new matrices cannot
give a smaller extended objective.
\end{proof}

For $K = -K$ and $\tau = 0$, the first variation controls moments of order $2p+1$.
The coefficient estimates are governed by the spectrum of $B_p$.
The same parameter $\lambda$ sets the first-order error
$\lambda\log(2N)$ and the negative curvature used in the support
bounds; the row cutoffs in \eqref{framework:eq-tensor-tail-exact}
remain available when more detailed spectral information is known.

\subsection{The inverse-trace and determinant endpoints}
 \label{framework:sec-endpoints}

The signed second variation explains the different information
provided by the endpoints. At $p = 1$, it retains the positive tensor
curvature used to bound coefficient supports. At $p = 0$, that curvature
cancels. At a global minimum with $\lambda>0$, finite replacements
force each positive-weight coefficient column to be a simplex vertex.
For the fixed-mass model below, the objective reduces to determinant
maximization.

\paragraph{The inverse-trace endpoint.}
At $p = 1$, $D_2 = 0$ because $\mathcal G_1(M) = \tr M$.
The term $D_1$ remains unless $Z_\xi(I-B^TB)^{1/2} = 0$; it must
therefore be retained in the general second-variation formula.

In the square model $m = n$, $\tau = 0$, the matrix
$B = M^{1/2}XW^{1/2}$ is orthogonal, so $D_1 = 0$ as well.
The curvature comparison therefore reduces to the positive tensor
norm and the negative entropy term in Lemma~\ref{cube:second}.
The first variation supplies the third inverse moment in
Theorem~\ref{framework:square}.

\paragraph{The determinant endpoint.}
When $p = 0$ and $R_\xi^2 = 0$,
$$
 -\frac12D^2\log\det M[M'_\xi,M'_\xi]
 = \frac12\tr(H_\xi^2) = \|R_\xi\|_{\F}^2 .
$$
Thus \eqref{framework:eq-cov-hess-exact} simplifies to
\begin{align}
 \label{framework:eq-cov-det-endpoint}
 \frac12(\Phi_0(A(s)))''\big|_0
 = -\|Z_\xi(I-B^TB)^{1/2}\|_{\F}^2 .
\end{align}
The positive curvature cancels completely. For $\lambda>0$, all such
variations must vanish. The two linear restrictions on a single column
therefore give at most two active coefficients per positive-weight
column at a local minimum.

At a global minimum for $p = 0$ and $\lambda>0$, each positive-weight
column can in fact have only one active coefficient. With all
other variables fixed, let $A_- = A-w_ix_ix_i^T\succeq0$.
The determinant of $A_-+w_iyy^T$ is
$\det A_-+w_i y^T\operatorname{adj}(A_-)y$.
The adjugate of a positive semidefinite matrix is positive
semidefinite (by continuity from positive definite matrices).
Thus this is a convex quadratic function of $y$ and some
generator has determinant at least its value at $x_i$.
If $q_i$ is not a vertex the entropy strictly decreases at this
atom, contradicting global minimality.

For the determinant model used in the cube upper bound, take $m = n$,
assume that $K = -K$ spans $\R^n$, and impose
$\tau>0$, $a = 0$, $\lambda>0$, and $\sum_iw_i = 1$.
Let $h_* = H(e_j)$, independent of $j$. The rank-one determinant identity
also allows each column of a determinant-maximizing representation to
be replaced by a generator without decreasing the determinant. The
resulting representation has weighted entropy $h_*$, the least possible
value. Therefore
\begin{align} \label{cube:determinant-endpoint-value}
 \min J_{0,\tau,0,\lambda}
 = -\frac12\log(\max\det A)+\lambda h_*.
\end{align}
Its minimizers are precisely the determinant maximizers whose
positive-weight coefficient columns are vertices. The number of columns
is prescribed; the assertion concerns the representations within those
columns. The determinant maximization and covariance completion of
Friedland-Naor-Youssef \cite{cube-fny} use this endpoint.

Appendix~\ref{cube:upper-appendix} records how the finite-replacement
inequality at this endpoint, together with the covariance completion
of Friedland-Naor-Youssef, gives the upper bound. The inverse-trace
endpoint supplies the variational input to the Gaussian lower bound.
The objectives belong to one spectral family; the covariance
regularization and weight constraints differ in the two arguments.

\section{Gaussian matrices under low-rank perturbations}
\label{cube:completion-section}

We first prove a Gaussian perturbation theorem independently of
the coefficient representation. The deterministic factor $T$ is
fixed before the Gaussian matrix is sampled. A single event will
control every low-rank correction and rounding error satisfying the
stated bounds, even if they depend on the whole sample.
The aggregate estimate \eqref{cube:aggregate} and the inverse
singular-value profile \eqref{cube:inverse-profile} are the two
conclusions needed later.

For a square $a\times a$ standard Gaussian matrix $G$, the corrected
Davidson-Szarek estimates \cite[Theorem 2.4]{cube-davidson-szarek},
\cite{cube-szarek-corrigendum} imply
\begin{align}
 \mathbb P\{s_j^\uparrow(G)<cj/\sqrt a\}& \le Ce^{-cj^2}
 &&(1 \le j \le a), \label{cube:sv-lower} \\
 \mathbb P\{s_j^\uparrow(G)>Cj/\sqrt a\}& \le Ce^{-cj^2}
 &&(1 \le j \le a/2). \label{cube:sv-upper}
\end{align}
The original normalization uses entry variance $1/a$ and decreasing
singular values. The corrected lower tail is $(C\delta)^{j^2}$ at
level $\delta j/a$; fix sufficiently small $\delta>0$. The corrected
upper tail is $\exp(-c\beta^2j\min\{j,a-j\})$ at level $\beta j/a$,
for sufficiently large fixed $\beta$. Rescaling by $\sqrt a$ gives
the displayed estimates; the restriction $j \le a/2$ is retained below.

\begin{lemma}[Trace bound on a random subspace] \label{cube:haar}
Let $D\succeq0$ be deterministic, let $v>0$ satisfy
$\|D\|_{\mathrm{op}}^2 \le v$, and put $\tau = \tr D^2$.
If $Q$ is a uniformly distributed $n\times p$ orthonormal frame,
where $1 \le p \le n$, then for every $u>0$,
\begin{align} \label{cube:haar-bound}
 \mathbb P\left\{\|DQ\|_{\mathrm F}^2>
 \frac{2p\tau+4pvu}{n}\right\}
 \le e^{-pu}.
\end{align}
\end{lemma}

\begin{proof}
Write a standard Gaussian $n\times p$ matrix as $B = QS^{1/2}$,
where $S = B^TB$. The frame $Q$ is uniform and independent of $S$:
the conditional law given $S$ is invariant under every left orthogonal
transformation, and the compact Stiefel manifold has a unique invariant
probability measure. Also $\mathbb ES = nI_p$. Jensen's inequality in
$S$, followed by the Gaussian quadratic-form formula, gives
$$
 \mathbb E\exp\left(\frac{n\|DQ\|_{\mathrm F}^2}{4v}\right)
 \le \mathbb E\exp\left(\frac{\|DB\|_{\mathrm F}^2}{4v}\right)
 = \det\left(I-\frac{D^2}{2v}\right)^{-p/2}
 \le \exp\left(\frac{p\tau}{2v}\right).
$$
The last inequality uses $-\log(1-x) \le 2x$ for $0 \le x \le 1/2$.
Exponential Markov inequality proves \eqref{cube:haar-bound}.
\end{proof}

\begin{lemma}[Trace on subspaces of a random subspace] \label{cube:haar-lower}
There are absolute constants $K,c,C>0$ with the following property.
Let $t>0$, let $D\succeq0$ be deterministic with $\|D\|_{\mathrm{op}}^2 \le t$,
and $\tau = \operatorname{tr}D^2$. Let $p \le n$ be an even positive
integer, put $q = p/2$, and let $\mathcal V$ be a uniformly distributed
$p$-dimensional subspace of $\mathbb R^n$. If $\tau \ge Kpt$, then,
outside an event of probability at most $Ce^{-cp^2}$,
\begin{align} \label{cube:haar-lower-bound}
 \operatorname{tr}(P_{\mathcal U}D^2) \ge \frac{p\tau}{16n}
 \quad\text{for every }\mathcal U\subseteq\mathcal V
 \text{ with }\dim\mathcal U = q.
\end{align}
\end{lemma}

\begin{proof}
First let $Q$ be a uniform $n\times q$ orthonormal frame. As in
Lemma~\ref{cube:haar}, write a standard Gaussian matrix as $B = QS^{1/2}$.
Jensen's inequality for the negative exponential gives
$$
 \mathbb E\exp\left(-\frac{n\|DQ\|_{\mathrm F}^2}{4t}\right)
 \le \mathbb E\exp\left(-\frac{\|DB\|_{\mathrm F}^2}{4t}\right)
 = \det\left(I+\frac{D^2}{2t}\right)^{-q/2}
 \le \exp\left(-\frac{q\tau}{8t}\right).
$$
The last inequality uses $\log(1+x) \ge x/2$ for $0 \le x \le 1/2$.
Therefore
\begin{align} \label{cube:haar-fixed-lower}
 \mathbb P\left\{\|DQ\|_{\mathrm F}^2 \le \frac{q\tau}{4n}\right\}
 \le \exp\left(-\frac{q\tau}{16t}\right).
\end{align}

Choose a uniform $n\times p$ orthonormal frame $W$ spanning $\mathcal V$.
Lemma~\ref{cube:haar}, with $u = \tau/t$, gives
\begin{align} \label{cube:haar-total}
 \operatorname{tr}(W^TD^2W) \le \frac{6p\tau}{n}
\end{align}
outside probability $e^{-p\tau/t}$.
Take a maximal $1/192$-separated family of $p\times q$ orthonormal
frames in the operator norm. A volume comparison in
$\mathbb R^{p\times q}$, using operator-norm balls, bounds its
cardinality by $385^{pq}$. Indeed, the disjoint balls of radius
$1/384$ lie in the ball of radius $1+1/384$.
Passing from a frame $V$ to $VV^T$ gives a $1/96$-net of the rank-$q$
orthogonal projections on $\mathbb R^p$.

For every fixed frame in this family, $WV$ is a uniform $n\times q$
frame. By \eqref{cube:haar-fixed-lower}, all the corresponding
projections $P_0$ satisfy
$$
 \operatorname{tr}(P_0W^TD^2W) \ge \frac{p\tau}{8n}
$$
except on an event of probability at most
$$
 \exp\left(\frac{\log385}{2}p^2-\frac{p\tau}{32t}\right).
$$
On \eqref{cube:haar-total}, moving from $P_0$ to any rank-$q$
projection $P$ within operator-norm distance $1/96$ changes the trace
by at most
$$
 \|P-P_0\|_{\mathrm{op}}\operatorname{tr}(W^TD^2W)
 \le \frac{p\tau}{16n}.
$$
This proves \eqref{cube:haar-lower-bound}. Taking $K$ sufficiently
large absorbs the net cardinality into the negative exponential.
The conclusion concerns the spectrum of the compression of $D^2$ to
$\mathcal V$ and is independent of the chosen frame $W$.
\end{proof}

\begin{theorem}[Gaussian matrices under low-rank perturbations] \label{cube:completion}
There are absolute constants $c,C>0$ and an integer $\kappa \ge 4$
with the following property.
Let $n \ge 2$, $d\in\{1,\ldots,n\}$, $t>0$, and let $T\succeq0$
be deterministic. Set
$$
 T_0 = \max(T,t^{-1/2}I),\quad D = T_0^{-1},\quad
 \ell_d = \log(en/d),
$$
where the maximum is defined by applying the scalar maximum to the
eigenvalues of $T$, and write the eigenvalues of $T_0$ as
$\tau_1 \le \cdots \le \tau_n$.
For a standard Gaussian $n\times n$ matrix $Z$, there is a Borel
event of probability at most $Ce^{-cd^2}$ outside which the following
holds simultaneously for all $F,\Delta$: if
\begin{align} \label{cube:completion-hypotheses}
 \begin{gathered}
 Y = ZT+F+\Delta\in\operatorname{GL}(n),\quad \operatorname{rank}F \le d, \\
 \|\Delta\|_{\mathrm{op}} \le \frac1{4\sqrt t},\quad
 \|Y^{-1}\|_{\mathrm F}^2 \le t,
 \end{gathered}
\end{align}
then
\begin{align} \label{cube:completion-conclusions}
 \tau_j^2 \ge \frac{j}{\kappa dt}\quad(1 \le j \le n),
\end{align}
and
\begin{align} \label{cube:aggregate}
 \operatorname{tr}D^2 \le Cdt.
\end{align}
Writing $\sigma_j$ for decreasing singular values, one also has
\begin{align} \label{cube:inverse-profile}
 \sigma_j(Y^{-1}) \le 
 C\min\left\{\sqrt{\frac tj},\frac{\sqrt{dt}}j\right\}
 \quad(1 \le j \le n).
\end{align}
In particular,
\begin{align} \label{cube:completion-weighted}
 \sum_{j = 1}^n j\sigma_j(Y^{-1})^2 \le Cdt \ell_d,
\end{align}
and
\begin{align} \label{cube:completion-nuclear}
 \|Y^{-1}\|_* \le C\sqrt{dt} \ell_d.
\end{align}
The exceptional event is determined by $Z,T,t,d$, independently of
the choices of $F$ and $\Delta$.

The pointwise inequalities \eqref{cube:completion-conclusions} alone
have the deterministic consequences
\begin{align} \label{cube:envelope-consequences}
 \operatorname{rank}(T_0-T) \le \kappa d,\quad
 \|D\|_{\mathrm{op}}^2 \le t,\quad
 \operatorname{tr}T_0^2 \ge \frac{n(n+1)}{2\kappa dt}.
\end{align}
The aggregate bound \eqref{cube:aggregate} is a further probabilistic
necessary condition for an admissible completion.
\end{theorem}

\begin{proof}
The proof has three steps: a pointwise lower bound for the
eigenvalues of $T_0$, an aggregate trace bound obtained on a random
subspace, and a dyadic approximation of the inverse.
All square Gaussian matrices inverted below are nonsingular almost
surely; include their null singularity events in the exceptional event.

\paragraph{The pointwise bound.}
Suppose \eqref{cube:completion-conclusions} fails, and choose a
violating index $q$ from the deterministic data $T,t,d$.
Since $\tau_q^2 \ge 1/t$, one has $q>\kappa d$.
By a right orthogonal change of coordinates, all the hypotheses are
preserved and we may assume that $T$ is diagonal in increasing order.
Split $Z = [Z_1,Z_2]$, with
$q$ columns in $Z_1$.
Let $P_q$ be an orthonormal frame for $(\operatorname{range}Z_2)^\perp$;
when $q = n$, take $P_q = I$.
Conditional on $Z_2$, the matrix $G_q = P_q^TZ_1$ is standard Gaussian.
For every $y\in\R^q$, the first $q$ coordinates of $Z^{-1}P_qy$
equal $G_q^{-1}y$. Therefore
$$
 \|DZ^{-1}P_qy\|_2 \ge \tau_q^{-1}\|G_q^{-1}y\|_2.
$$
Applying the singular-value min-max principle to these inverses,
and then \eqref{cube:sv-upper}, shows that outside an event of
probability $Ce^{-cd^2}$,
$$
 s_{2d}^\uparrow(ZT_0) \le \tau_q s_{2d}^\uparrow(G_q)
 \le \frac{Cd\tau_q}{\sqrt q}.
$$
In detail, restriction to the isometric image of $P_q$ cannot
increase the $2d$-th largest singular value of $DZ^{-1}$.
The preceding pointwise lower bound makes that restricted singular
value at least $\tau_q^{-1}$ times the $2d$-th largest singular
value of $G_q^{-1}$. Taking reciprocals gives the displayed
comparison of the $2d$-th \emph{smallest} singular values.
Here $2d \le q/2$ follows by taking $\kappa \ge 4$.
Since $ZT^2Z^T\preceq ZT_0^2Z^T$, the same upper bound holds
for $s_{2d}^\uparrow(ZT)$, even if $T$ is singular.
Choose a $2d$-dimensional right singular subspace of $ZT$
on which its norm is at most this bound.
Intersecting it with $\ker F$ leaves at least $d$ dimensions.
On this intersection,
$$
 \|Yx\|_2
 \le \left(\frac{Cd\tau_q}{\sqrt q}
 +\frac1{4\sqrt t}\right)\|x\|_2
 <\left(\frac C{\sqrt\kappa}+\frac14\right)
 \sqrt{d/t} \|x\|_2.
$$
Choose $\kappa$ so large that the final coefficient is at most $1/2$.
Then $s_d^\uparrow(Y) \le \frac12\sqrt{d/t}$, and hence
$\|Y^{-1}\|_{\mathrm F}^2 \ge d/s_d^\uparrow(Y)^2 \ge 4t$, a contradiction.
Thus the spectral inequalities are necessary outside one
exceptional event. No union over the possible violating indices
is needed.

Suppose that the pointwise inequalities hold. At most $\kappa d$
eigenvalues of $T$ are raised in forming $T_0$, because every raised
eigenvalue of $T_0$ has square $1/t$. Summing the pointwise inequalities
proves the remaining assertion of \eqref{cube:envelope-consequences}.
Put
$$
 h = (\kappa+1)d,\quad \tau = \operatorname{tr}D^2.
$$
For every completion satisfying \eqref{cube:completion-hypotheses},
let $V = Y-\Delta$. Since $\|Y^{-1}\Delta\|_{\mathrm{op}} \le 1/4$,
$V$ is invertible and
\begin{align} \label{cube:inverse-correction}
 \|V^{-1}\|_{\mathrm F} \le 2\sqrt t,\quad
 \operatorname{rank}\big(V^{-1}-DZ^{-1}\big) \le h.
\end{align}
Indeed, $V-ZT_0 = F+Z(T-T_0)$ has rank at most $h$.
Multiplication by the two inverses preserves this rank bound.

Fix a singular-value decomposition
$$
 Z = \sum_{j = 1}^n s_j^\uparrow u_jw_j^T,\quad s_1^\uparrow \le \cdots \le s_n^\uparrow.
$$

\paragraph{The trace bound.}
If $n<4h$, then $\tau \le nt<4ht$, as required.
Otherwise set $p = 2h$ and $q = h$, so that $p \le n/2$.
The number $\tau$ is deterministic. If $\tau \le Kpt$ for a
sufficiently large absolute $K$, the desired bound already holds.
Assume therefore that $\tau>Kpt$.
Let $U_p,W_p$ contain the first $p$ left and right singular vectors.
The span of the columns of $W_p$ is a uniform $p$-dimensional subspace.
Lemma~\ref{cube:haar-lower} applies to this subspace; its conclusion
is independent of the signs or basis used for the singular vectors.
Together with \eqref{cube:sv-upper}, it gives, outside probability
$Ce^{-cp^2}$,
$$
 s_p^\uparrow \le C_{\mathrm{sv}}p/\sqrt n,\quad
 \min_{\substack{P = P^T = P^2\\\operatorname{rank}P = q}}
 \operatorname{tr}(PW_p^TD^2W_p) \ge \frac{p\tau}{16n}.
$$
No independence between these two events is used.

Now
$$
 DZ^{-1}U_p = DW_p\operatorname{diag}((s_1^\uparrow)^{-1},\ldots,(s_p^\uparrow)^{-1}).
$$
The product of this matrix with its transpose dominates
$(s_p^\uparrow)^{-2}DW_pW_p^TD$. Thus the decreasing singular values of
$DZ^{-1}U_p$ are at least those of $DW_p$ divided by $s_p^\uparrow$.
Restricting \eqref{cube:inverse-correction} to $U_p$ still gives a
rank-$h$ difference. The best rank-$h$ approximation inequality yields
$$
 \begin{aligned}
 \|V^{-1}\|_{\mathrm F}^2
 & \ge \|V^{-1}U_p\|_{\mathrm F}^2
 \ge \sum_{j = h+1}^{p}\sigma_j(DZ^{-1}U_p)^2 \\
 & \ge (s_p^\uparrow)^{-2}\sum_{j = h+1}^{p}\sigma_j(DW_p)^2
 \ge \frac{\tau}{16C_{\mathrm{sv}}^2p}.
 \end{aligned}
$$
Here the sum of the $q$ smallest eigenvalues of $W_p^TD^2W_p$ is the
minimum over rank-$q$ projections in the preceding display.
Choose $K>64C_{\mathrm{sv}}^2$, also large enough for
Lemma~\ref{cube:haar-lower}. The last bound exceeds $4t$, contradicting
\eqref{cube:inverse-correction}. This proves \eqref{cube:aggregate}
outside an event of probability $Ce^{-cd^2}$, simultaneously for all
completions. The cases $\tau \le Kpt$ and $\tau>Kpt$ were determined by
$T,t,d$ alone.

\paragraph{The inverse singular values.}
We may now assume the deterministic bound $\operatorname{tr}D^2 \le Cdt$.
If $n \le 8h$, the Frobenius budget gives
$\sigma_j(Y^{-1}) \le \sqrt{t/j} \le C\sqrt{dt}/j$ for every $j$,
so \eqref{cube:inverse-profile} follows.
Suppose $n>8h$ and retain the increasing Gaussian singular values.
A union bound in \eqref{cube:sv-lower} gives
\begin{align} \label{cube:sv-event}
 s_j^\uparrow \ge cj/\sqrt n\quad(d \le j \le n)
\end{align}
with failure probability at most $Ce^{-cd^2}$.
For a dyadic index $p = 2^ad<n$, set $p' = \min(2p,n)$.
The span of the first $p'$ right singular vectors is uniform.
Lemma~\ref{cube:haar}, with $v = t$, $u = d$ and the trace bound, gives
\begin{align} \label{cube:dyadic-trace}
 \sum_{j \le p'}\|Dw_j\|_2^2 \le \frac{Cpdt}{n}
\end{align}
simultaneously for these $p$. The sum of the failure probabilities is
at most $\sum_{a \ge 0}e^{-2^ad^2} \le Ce^{-cd^2}$.

For a dyadic block set
$$
 R_p = \sum_{j = p+1}^{p'}(s_j^\uparrow)^{-1}(Dw_j)u_j^T.
$$
Then
$$
 \|R_p\|_{\mathrm F}^2
 = \sum_{j = p+1}^{p'}(s_j^\uparrow)^{-2}\|Dw_j\|_2^2
 \le Cdt/p.
$$
For a fixed dyadic $p$, the blocks with starting indices
$p,2p,4p,\ldots$ are orthogonal in Frobenius inner product, because
their right factors are disjoint groups of the vectors $u_j$.
Their sum $E_p$ consequently satisfies
\begin{align} \label{cube:inverse-tail}
 DZ^{-1} = A_p+E_p,\quad \operatorname{rank}A_p \le p,\quad
 \|E_p\|_{\mathrm F}^2 \le Cdt/p.
\end{align}
By \eqref{cube:inverse-correction}, $V^{-1}-E_p$ has rank at most $h+p$.
Hence
\begin{align} \label{cube:rank-tail}
 \sum_{j>h+p}\sigma_j(V^{-1})^2 \le Cdt/p.
\end{align}
If $j \ge 8h$, choose a dyadic $p$ with $j/8 \le p \le j/4$.
Then $h+p \le 3j/8$, and \eqref{cube:rank-tail} implies
$\sigma_j(V^{-1}) \le C\sqrt{dt}/j$.
For $j<8h$, the same inequality, with a larger absolute constant,
follows from $\|V^{-1}\|_{\mathrm F} \le 2\sqrt t$.
Finally,
$$
 Y^{-1} = (I-Y^{-1}\Delta)V^{-1},\quad
 \|I-Y^{-1}\Delta\|_{\mathrm{op}} \le 5/4.
$$
This proves the second bound in \eqref{cube:inverse-profile}; the
first follows directly from the prescribed Frobenius budget.
Multiplying the squared profile by $j$ and summing gives
$$
 \sum_{j = 1}^n j\sigma_j(Y^{-1})^2
 \le Cdt+Cdt\sum_{j = d+1}^n\frac1j
 \le Cdt\log(en/d),
$$
which proves \eqref{cube:completion-weighted}.
Summing the first $d$ singular values by Cauchy-Schwarz and the
remaining ones harmonically gives
\begin{align} \label{cube:rank-nuclear}
 \|Y^{-1}\|_*
 \le \sqrt d \|Y^{-1}\|_{\mathrm F}
 +C\sqrt{dt}\sum_{j = d+1}^n\frac1j
 \le C\sqrt{dt}\log(en/d),
\end{align}
which proves \eqref{cube:completion-nuclear}.
All the exceptional events are Borel conditions on the Gaussian
matrix and deterministic data. The rank estimates hold for every
$F,\Delta$ on the same events, proving the asserted uniformity.
\end{proof}

The same profile gives, for every fixed $1<s<2$,
\begin{align} \label{cube:completion-schatten}
 \|Y^{-1}\|_{S_s} \le C_s\sqrt t d^{1/s-1/2},
 \quad \|A\|_{S_s} = \left(\sum_j\sigma_j(A)^s\right)^{1/s}.
\end{align}
Indeed, H\"older's inequality bounds the sum of the first $d$ singular
values to the power $s$ by $d^{1-s/2}t^{s/2}$, while the remaining sum
is bounded by $C_s(dt)^{s/2}d^{1-s}$.

\begin{corollary}[A necessary bound on the perturbation rank] \label{cube:rank-energy}
Under the hypotheses of Theorem~\ref{cube:completion}, outside its
exceptional event every admissible completion satisfies
\begin{align} \label{cube:rank-energy-general}
 n^2 \le Cdt\left(\operatorname{tr}T^2+\frac nt\right).
\end{align}
In particular, there is an absolute constant $C_{\mathrm{en}}>0$
such that, if $\operatorname{tr}T^2 \le t$ and $t \ge \sqrt n$, then
\begin{align} \label{cube:rank-energy-bound}
 n^2 \le C_{\mathrm{en}}dt^2.
\end{align}
\end{corollary}

\begin{proof}
Cauchy-Schwarz applied to the eigenvalues of $T_0$ gives
$n^2 \le (\operatorname{tr}T_0^2)(\operatorname{tr}D^2)$.
Use \eqref{cube:aggregate} and
$\operatorname{tr}T_0^2 \le \operatorname{tr}T^2+n/t$.
The last assertion follows from $n/t \le t$.
\end{proof}

The bounds \eqref{cube:rank-energy-bound},
\eqref{cube:completion-weighted} and
\eqref{cube:completion-nuclear} are sharp up to absolute constants
for arbitrary low-rank perturbations.

\begin{proposition}[Sharpness for a scalar multiple of a Gaussian matrix] \label{cube:completion-sharpness}
There are absolute constants $A,c,C>0$ such that the following holds.
Let $1 \le d \le n/4$, put
$$
 t = An/\sqrt d,\quad T = \sqrt{t/n} I,
$$
and let $Z$ be standard Gaussian. With probability at least
$1-Ce^{-cd^2}$, there exists $F$, possibly depending on $Z$, with
$\operatorname{rank}F \le d$ such that $Y = ZT+F$ is invertible and
\begin{align} \label{cube:completion-sharpness-bound}
 \|Y^{-1}\|_{\mathrm F}^2 &\le t,\quad
 \|Y^{-1}\|_* \ge c\sqrt{dt}\log(en/d),\\
 \sum_{j=1}^n j\sigma_j(Y^{-1})^2
 &\ge cdt\log(en/d). \notag
\end{align}
Furthermore $\operatorname{tr}T^2 = t$ and
$\operatorname{tr}T_0^{-2} = A^{-2}dt$.
\end{proposition}

\begin{proof}
The estimates \eqref{cube:sv-lower}-\eqref{cube:sv-upper}, followed by
union bounds, give
$$
 s_j^\uparrow(Z) \ge cj/\sqrt n\quad(d<j \le n),\quad
 s_j^\uparrow(Z) \le Cj/\sqrt n\quad(d<j \le n/2)
$$
outside probability $Ce^{-cd^2}$.
Write $Z = U\operatorname{diag}(s_j^\uparrow)V^T$ in increasing order and set
$$
 F = U\operatorname{diag}(b-\sqrt{t/n} s_1^\uparrow,\ldots,
 b-\sqrt{t/n} s_d^\uparrow,0,\ldots,0)V^T,
 \quad b = 2\sqrt{d/t}.
$$
Then $\operatorname{rank}F \le d$ and
$$
 \|Y^{-1}\|_{\mathrm F}^2
 = d/b^2+\frac nt\sum_{j = d+1}^n s_j^\uparrow(Z)^{-2}
 \le \frac t4+\frac{Cn^2}{td}
 = \frac t4+\frac C{A^2}t \le t
$$
when $A$ is a sufficiently large absolute constant.
On the other hand,
$$
 \|Y^{-1}\|_*
 \ge \sqrt{\frac nt}\sum_{j = d+1}^{\lfloor n/2\rfloor}s_j^\uparrow(Z)^{-1}
 \ge \frac{cn}{\sqrt t}\log(en/d)
 \ge c\sqrt{dt}\log(en/d),
$$
where the last constant absorbs the fixed value of $A$.

For $1\le j\le\lfloor n/2\rfloor-d$, the unchanged singular
directions with indices $d+1,\ldots,d+j$ give
$$
 \sigma_j(Y^{-1})
 \ge \sqrt{\frac nt}\frac{1}{s_{d+j}^\uparrow(Z)}
 \ge \frac{cn}{\sqrt t(d+j)}.
$$
Consequently,
$$
 \sum_{j=1}^n j\sigma_j(Y^{-1})^2
 \ge \frac{cn^2}{t}
 \sum_{j=1}^{\lfloor n/2\rfloor-d}\frac{j}{(d+j)^2}
 \ge \frac{cn^2}{t}\log(en/d)
 \ge cdt\log(en/d).
$$
Indeed, since $d\le n/4$, the first $d$ terms of the sum
contribute an absolute constant, and the remaining terms
give the harmonic contribution. The last inequality uses
$t=An/\sqrt d$, with $A$ fixed.

Since $t^2 \ge n$, one has $T_0 = T$, and the trace identities follow.
\end{proof}

Thus the logarithmic factors in \eqref{cube:completion-weighted}
and \eqref{cube:completion-nuclear} cannot be removed for arbitrary
rank-$d$ perturbations. This does not imply sharpness of the
logarithmic factor in Theorem~\ref{cube:main}: the perturbations
in that theorem satisfy additional coefficient and inverse-moment bounds.

\section{The Banach-Mazur radius of the cube} \label{cube:chapter}

We apply the inverse-moment and coefficient estimates of
Theorem~\ref{framework:square} to a symmetric Gaussian polytope.
The minimizing representation depends on the entire sample. The
probabilistic argument must therefore exclude all representations
allowed by those estimates, including ones chosen after the sample
is known.

\subsection{Statement and background} \label{intro:cube-history}

For origin-symmetric convex bodies $K,Q\subset\R^n$, let
$$
 d_{\mathrm{BM}}(K,Q)
 = \inf\{a \ge 1:TQ\subseteq K\subseteq aTQ,
 T\in\mathrm{GL}(n)\},
 \quad
 R_\infty(n) = \sup_K d_{\mathrm{BM}}(K,B_\infty^n).
$$
Thus $R_\infty(n)$ measures the largest distortion needed to approximate
an $n$-dimensional origin-symmetric body by a parallelepiped. Polarity
relates this problem to approximation by the cross-polytope:
\begin{align} \label{cube:polarity}
 d_{\mathrm{BM}}(K,B_1^n) = d_{\mathrm{BM}}(K^\circ,B_\infty^n).
\end{align}
Indeed, polarizing $TB_1^n\subseteq K\subseteq aTB_1^n$ gives
$a^{-1}T^{-T}B_\infty^n\subseteq K^\circ\subseteq T^{-T}B_\infty^n$,
and the converse follows by polarizing again.

\begin{theorem}[Distance from the cube] \label{cube:main}
There are absolute constants $c,C>0$ such that the following
holds for every $n\ge2$. Let $g_1,\ldots,g_{2n}$ be independent
standard Gaussian vectors in $\mathbb R^n$, and set
$$
 K = \operatorname{conv}\{\pm g_1,\ldots,\pm g_{2n}\},\quad P = K^\circ.
$$
Then $K$ and $P$ are convex bodies almost surely, and
\begin{align} \label{cube:mainbound}
 \mathbb P\left\{
 d_{\mathrm{BM}}(P,B_\infty^n)
 \ge \frac{c n^{2/3}}{(\log n)^{5/6}}
 \right\} \ge 1-Ce^{-cn}.
\end{align}
Consequently,
$$
 R_\infty(n) \ge \frac{c n^{2/3}}{(\log n)^{5/6}}
 \quad(n \ge 2).
$$
\end{theorem}

Together with the upper bound of Friedland-Naor-Youssef
\cite{cube-fny}, this gives
$$
 \frac{cn^{2/3}}{(\log n)^{5/6}}\le R_\infty(n)\le2n^{2/3},
 \quad
 \lim_{n\to\infty}\frac{\log R_\infty(n)}{\log n} = \frac23.
$$
The remaining gap is logarithmic. Section~\ref{framework:sec-endpoints}
relates the determinant objective in the upper bound to the
inverse-trace objective used here; Appendix~\ref{cube:upper-appendix}
records the connection with the upper-bound argument.

\paragraph{Earlier bounds.}
The question of maximal distance to the cube was posed by
Gr\"unbaum \cite[Section~6, Remark~(ii)]{cube-grunbaum}.
The related question of Pe{\l}czy\'nski on asymptotic centers
\cite[Section~XVI]{cube-pelczynski} asks, in particular, whether
$R_\infty(n)/\sqrt n$ tends to infinity. Szarek \cite{cube-szarek90}
proved the lower bound $c\sqrt n\log n$. Tikhomirov
\cite[Theorem~A]{cube-tikhomirov} obtained the exponent $5/9$,
up to logarithmic factors; the author \cite{cube-friedland47}
subsequently obtained $4/7$. The exponent $5/8$ was reached
independently in \cite{cube-friedland58,cube-hmadi}.
Hmadi's Theorem~1.1 uses $n^3$ Gaussian generators and gives
$cn^{5/8}(\log n)^{-5/8}$ with probability at least $1-2/n$.
The revised Theorem~1.1 of \cite{cube-friedland58} uses
$\lceil(1+\delta)n\rceil$ generators and gives
$c_\delta n^{5/8}(\log n)^{-1/8}$ with exponentially high probability,
for each fixed $\delta>0$ and sufficiently large $n$.

The elementary upper bound is $n$. Bourgain-Szarek \cite{cube-bs}
improved it to $Cn\exp(-c\sqrt{\log n})$.
Subsequently, Szarek-Talagrand \cite{cube-st} and Giannopoulos
\cite{cube-giannopoulos} obtained the exponents $7/8$ and $5/6$,
respectively. Youssef \cite[Theorem~3.7]{cube-youssef} gave the explicit
bound $(2n)^{5/6}$ through restricted invertibility.
Friedland-Naor-Youssef \cite{cube-fny} obtained $2n^{2/3}$ by
determinant maximization and covariance completion.

We use the random-polytope method of Gluskin \cite{cube-gluskin},
developed by Szarek and Tikhomirov to study distance to a prescribed
cross-polytope. Theorem~\ref{cube:main} retains a linear
number of generators and exponentially small failure probability.

\subsection{Proof strategy} \label{cube:strategy}

A small distance to the cross-polytope would give an inscribed basis
with small sum of Euclidean lengths of its dual vectors.
Theorem~\ref{framework:square} then gives one representation with both
controlled inverse moments and few coefficients. We divide the
generators according to their energy in its inverse covariance.
The ordinary coefficients can be rounded to a finite set; the
exceptional coefficients are left unrestricted at this stage.

Two Gaussian arguments are used. First, for each fixed record of
ordinary coefficients and exceptional indices, conditioning on the
exceptional columns gives
$$
 Y = ZT+F+\Delta,
$$
where $T$ is deterministic, $Z$ is standard Gaussian, $F$ has low
rank and $\Delta$ is a small rounding error.
Theorem~\ref{cube:completion} controls the inverse uniformly over all
$F,\Delta$ satisfying the prescribed rank, error and energy bounds.
Its rank obstruction excludes a small exceptional set. For a larger
set, its inverse singular-value estimate improves the collective
coefficient bound enough to permit a further discretization.

Second, the conditional coordinates separate $Z$ into independent
upper and lower blocks. The lower block determines a right test
frame. The third inverse moment supplies a left test frame, which
may depend on the full sample and is therefore covered by a finite
net. Every admissible representation forces a cancellation of the
upper Gaussian block between these frames. A conditional small-ball
estimate excludes this cancellation.

Lemma~\ref{cube:threshold} chooses the energy threshold so that the
record counts fit the Gaussian probability estimates. Each probability
is computed for a record fixed in advance, without conditioning on
the existence of a minimizing representation that produces it.

\subsection{The Gaussian model and reduction to an inscribed basis}

The auxiliary constants are chosen in the following order:
$\alpha>0$ sufficiently small, $\gamma>0$ sufficiently small relative
to $\alpha$, and $B_{*}>0$ sufficiently large. Their required
inequalities are collected at the end of the proof. For an energy
budget $t$, put
\begin{align} \label{cube:parameters}
 \begin{gathered}
 m = 2n,\quad N = 2m,\quad \Gamma = [g_1 \cdots g_m],\quad
 L = \log n,\quad u_0 = B_{*}L^3, \\
 k = \left\lfloor\frac{\gamma t u_0^{3/2}}{L^2}\right\rfloor
 = \lfloor\gamma B_{*}^{3/2}tL^{5/2}\rfloor.
 \end{gathered}
\end{align}
The signed generators are $y_j = g_j$ and $y_{m+j} = -g_j$ for $j \le m$.
Throughout the argument assume
\begin{gather}
 \sqrt n \le t \le n, \label{cube:energy-range} \\
 C_{\mathrm{en}}kt^2 \le n^2/2, \label{cube:small-constant}
\end{gather}
where $C_{\mathrm{en}}$ is the constant in
Corollary~\ref{cube:rank-energy}. All asymptotic statements below hold
for sufficiently large $n$, depending only on
$\alpha,\gamma,B_{*}$, and uniformly over this range of $t$.
In particular,
\begin{align} \label{cube:scales}
 k = o(n),\quad n = o(k^2),\quad
 \frac{t^2u_0L^2}{k^2} = \frac{1+o(1)}{\gamma^2B_{*}^2}.
\end{align}
Indeed, $k\asymp\gamma B_{*}^{3/2}tL^{5/2}$, and
\eqref{cube:small-constant} gives $t^3L^{5/2} \le C_{\gamma,B_{*}}n^2$.
Consequently $k \le C_{\gamma,B_{*}}n^{2/3}L^{5/3} = o(n)$, while
$t \ge \sqrt n$ gives $k^2/n \ge c\gamma^2B_{*}^3L^5$.

\paragraph{Choice of scales.}
Lemma~\ref{cube:threshold} chooses the threshold $\sigma$ and the
number $r$ of exceptional generators. In the large exceptional
case, $r\asymp\gamma t\sigma^{3/2}/L^2$ and $r>k$.
The logarithmic count of ordinary records, divided by the probability
exponent $r^2$, is at most $CL^6/(\gamma^2\sigma^2)$.
For each record, let $b$ be the Gaussian test dimension in
\eqref{cube:right-dimension}. The coefficient and frame net counts,
divided by $b$, are bounded by
$$
 C\left(\frac\gamma\alpha+\frac{L^3}{\alpha\sigma}\right).
$$
After $\alpha$ is fixed, choosing $\gamma$ small controls the first
term; imposing $\sigma\ge u_0 = B_*L^3$ and choosing $B_*$ large
controls the second. In the small exceptional case, the total
logarithmic record count is
$(O(\gamma^{-2}B_*^{-2})+o(1))k^2$, against a probability bound
$Ce^{-ck^2}$ for each record.

Corollary~\ref{cube:rank-energy} requires $kt^2\gtrsim n^2$ for a
completion of rank at most $k$. Condition~\eqref{cube:small-constant}
will contradict this requirement. Since
$k\asymp\gamma B_*^{3/2}tL^{5/2}$, the limiting balance is
$t^3L^{5/2}\asymp n^2$, giving
$t = a_*n^{2/3}/L^{5/6}$ for a sufficiently small absolute $a_*>0$.
The final assembly verifies that this choice meets all the conditions.

\begin{lemma}[A simultaneous Gaussian event] \label{cube:good}
There are absolute constants $c_0>0$ and $C_0 \ge 2$ for which an event
$\mathcal G$ of probability at least $1-Ce^{-cn}$ has the following
properties:
\begin{align}
 \|\Gamma\|_{\mathrm{op}}^2& \le C_0n, \label{cube:goodop} \\
 \frac1m\sum_{j = 1}^m|\langle g_j,u\rangle|
 & \ge c_0\|u\|_2\quad(u\in\mathbb R^n), \label{cube:goodmean} \\
 \frac n4 I&\preceq\Gamma_I^T\Gamma_I\preceq4nI
 \quad\left(1 \le |I| \le d_{*}\right),\quad
 d_{*} = \left\lceil\frac{4C_0n}{u_0}\right\rceil.
 \label{cube:goodgram}
\end{align}
Here $\Gamma_I$ is the matrix consisting of the columns indexed by $I$.
\end{lemma}

The proof is given in Appendix~\ref{cube:elementary-appendix}.

Recall that an inscribed basis has $Xe_i\in K$ for every $i$, and
that $\beta(X) = \sum_i\|X^{-T}e_i\|_2$ as in
Section~\ref{framework:section}.
\begin{lemma}[Reduction to the dual basis] \label{cube:reduction}
On $\mathcal G$,
\begin{align} \label{cube:basisreduction}
 d_{\mathrm{BM}}(K,B_1^n)
 = \inf_{\substack{X\in\mathrm{GL}(n) \\Xe_i\in K}}
 \max_{j \le m}\|X^{-1}g_j\|_1
 \ge c_0\inf_{\substack{X\in\mathrm{GL}(n) \\Xe_i\in K}}\beta(X).
\end{align}
\end{lemma}

\begin{proof}
The condition $XB_1^n\subseteq K$ is equivalent to $Xe_i\in K$ for
every $i$. Since $K$ is the convex hull of the signed generators,
the other inclusion $K\subseteq aXB_1^n$ is equivalent to
$\max_j\|X^{-1}g_j\|_1 \le a$. This proves the equality. With
$v_i = X^{-T}e_i$, averaging over the generators gives
$$
 \max_j\|X^{-1}g_j\|_1
 \ge \sum_i\frac1m\sum_j|\langle v_i,g_j\rangle|
 \ge c_0\sum_i\|v_i\|_2.
$$
Taking the infimum proves the inequality.
\end{proof}

It remains to exclude inscribed bases with $\beta(X) \le t/2$,
outside probability $Ce^{-cn}$.

\subsection{Sparse coefficient representations} \label{cube:variational-section}

Suppose that an inscribed basis satisfies $\beta(X_0) \le t/2$.
Apply Theorem~\ref{framework:square} to the signed generating family.
It gives a minimum $(p,W)$ of \eqref{cube:objective}, with
$X = [\sum_jp_{ji}y_j]_i$, $Y = XW^{1/2}$, and $M = (YY^T)^{-1}$,
satisfying \eqref{cube:budgets} and \eqref{cube:cubic}, with
$\tr M\le t$, and the simultaneous row estimates
\eqref{framework:square-col}-\eqref{framework:square-collect}.
All subsequent estimates use this representation, with no added
covariance term. The separate-column bound controls coefficients in
rows of moderate energy; the collective bound will control the
remaining rows after an estimate of the inverse spectrum.

On $\mathcal G$, fix $\sigma\ge u_0$ and a set
$J\subseteq\{1,\ldots,m\}$ such that
\begin{align} \label{cube:energy-threshold}
 g_j^TMg_j \ge \frac{\sigma t}{2}\quad(j\in J),
 \quad g_j^TMg_j \le \sigma t\quad(j\notin J).
\end{align}
Write
$$
 r = |J|,\quad E = \operatorname{span}\{g_j:j\in J\},\quad
 e = \tr(P_EM),\quad \widetilde J = J\cup\{m+j:j\in J\}.
$$
An active coefficient is called exceptional if its row belongs to
$\widetilde J$, and ordinary otherwise. Let $S_1$ and $S_0$
be the respective numbers of active coefficients.
The operator-norm bound on $\mathcal G$ gives
$$
 \frac{r\sigma t}{2}
 \le \sum_{j\in J}g_j^TMg_j
 \le \tr(M\Gamma\Gamma^T) \le C_0nt.
$$
Thus $r \le 2C_0n/\sigma \le 2C_0n/u_0 \le d_{*}$, so the Gram bounds on
$\mathcal G$ apply to $J$. They imply
\begin{align} \label{cube:exceptional-energy}
 \dim E = r,\quad
 e \ge \frac{r\sigma t}{8n},\quad
 \tr(P_EM^3) \le \frac{64rt}{n} \le \frac{512e}{\sigma}.
\end{align}
Indeed, $(n/4)P_E\preceq\Gamma_J\Gamma_J^T\preceq4nP_E$.
Taking traces against $M$ gives the lower bound for $e$;
taking traces against $M^3$ and applying \eqref{cube:cubic} to
each $g_j$, $j\in J$, gives the third-moment bound.
For $J = \varnothing$, take $E = \{0\}$ and $e = r = 0$.

\begin{proposition}[Sparse ordinary and exceptional coefficients]
 \label{cube:support}
There is an absolute constant $C$ such that every minimizing
representation above satisfies, with $\mu_1 \ge \cdots \ge \mu_n$ the
eigenvalues of $M$,
\begin{align} \label{cube:support-estimates}
 S_0 \le Ct^2\sigma L,
 \quad
 S_1 \le n+r^2+C\left(neL\sum_{\ell = 1}^n\ell\mu_\ell\right)^{1/2}.
\end{align}
\end{proposition}

\begin{proof}
For ordinary rows, \eqref{framework:square-col} gives
$$
 S_0 \le 2n+C(\log N)\sigma t^2 \le Ct^2\sigma L,
$$
using $t \ge \sqrt n$ and $N = 4n$.
For the exceptional rows, the Gram bound implies
$$
 \rho_{\widetilde J}
 = \sum_{j\in\widetilde J}y_j^TMy_j
 = 2\operatorname{tr}(M\Gamma_J\Gamma_J^T) \le 8ne.
$$
Their span has dimension $r$. The first inequality in
\eqref{framework:square-collect} gives
\begin{align} \label{cube:collective-dimension}
 S_1& \le n+r^2+
 C\left(\rho_{\widetilde J}\log N
 \sum_{\ell = 1}^n\ell\mu_\ell\right)^{1/2}\notag\\
 & \le n+r^2+C\left(neL\sum_{\ell = 1}^n\ell\mu_\ell\right)^{1/2}.
\end{align}
\end{proof}

\subsection{Adaptive thresholds and finite records} \label{cube:threshold-section}

We now choose the energy threshold. Raising it increases the
ordinary coefficient bound and decreases the exceptional set.
Lemma~\ref{cube:threshold} balances these two quantities. The choice
depends on the minimizing representation, but the resulting records
belong to a finite family fixed before the probability estimates.

\begin{lemma}[Adaptive threshold] \label{cube:threshold}
On $\mathcal G$, let $(p,W)$ be a minimum with value at most $t$.
There are a dyadic number $\sigma$ and a set $J\subseteq\{1,\ldots,m\}$,
of cardinality $r$, such that
\begin{gather}
 u_0 \le \sigma \le 2C_0n,\quad r \le C_0n/u_0,
 \label{cube:threshold-range} \\
 g_j^TMg_j \ge \frac{\sigma t}{2}\quad(j\in J),
 \quad
 g_j^TMg_j \le \sigma t\quad(j\notin J).
 \label{cube:threshold-energies}
\end{gather}
Furthermore, one of the following holds:
\begin{align}
 &\sigma = \sigma_0 := 2^{\lceil\log_2u_0\rceil},\quad r \le k;
 \label{cube:case-small} \\
 &r>k,\quad
 2^{-3/2}\frac{\gamma t\sigma^{3/2}}{L^2}
 \le r \le 2\frac{\gamma t\sigma^{3/2}}{L^2}.
 \label{cube:case-large}
\end{align}
\end{lemma}

\begin{proof}
Let $\eta_j = g_j^TMg_j/t$ and
$R(u) = \gamma tu^{3/2}/L^2$. The trace bound and
\eqref{cube:goodop} give
$$
 \sum_j\eta_j
 = t^{-1}\operatorname{tr}(M\Gamma\Gamma^T) \le C_0n.
$$
If $\#\{j:\eta_j>u_0\} \le R(u_0) = \gamma B_{*}^{3/2}tL^{5/2}$,
take this set as $J$ and choose $\sigma = \sigma_0$. Then $r \le k$, and
$u_0 \le \sigma_0<2u_0$ proves the energy inequalities.

Otherwise define
$$
 \tau = \inf\{u \ge u_0:\#\{j:\eta_j>u\} \le R(u)\}.
$$
The set is nonempty. Right continuity of the counting function
and the strict inequality at $u_0$ give $\tau>u_0$. Its left limit
at $\tau$ is $\#\{j:\eta_j\ge\tau\}$; continuity of $R$ therefore gives
$$
 \#\{j:\eta_j>\tau\} \le R(\tau)
 \le \#\{j:\eta_j \ge \tau\}.
$$
Take every index with energy strictly larger than $\tau$, and then
enough indices with energy equal to $\tau$ to make
$r = \lceil R(\tau)\rceil$. Ties can be settled by increasing index.
Set $\sigma = 2^{\lceil\log_2\tau\rceil}$. Since
$\tau \le \sigma<2\tau$ and
$R(\tau) \le r \le 2R(\tau)$ for large $n$, both
\eqref{cube:threshold-energies} and \eqref{cube:case-large} follow.

Every selected normalized energy $\eta_j$ is at least $u_0$, so
$ru_0 \le \sum_j\eta_j \le C_0n$. Also
$\tau \le \max_j\eta_j \le C_0n$, which proves the stated range
of $\sigma$. The small-rank construction satisfies the same range
for sufficiently large $n$.
\end{proof}

\subsubsection{Ordinary coefficient records}
 \label{cube:record-section}

For the minimum under consideration, combine the two signs by setting
$$
 C_{ji} = \sqrt{W_i} (p_{ji}-p_{m+j,i})
 \quad(1 \le j \le m, 1 \le i \le n).
$$
Then $Y = \Gamma C$, and
$\sum_i\|Ce_i\|_1^2 \le \tr W \le t$.
For a set $J$ satisfying \eqref{cube:energy-threshold}, put $A = C_J$.
Round every entry of $C_{J^c}$ toward zero to a multiple of $n^{-4}$,
and call the resulting matrix $B$.
Keep the original order of the rows in $A$ and $B$.
Rounding decreases the absolute values and support sizes, so on
$\mathcal G$ we have
\begin{align} \label{cube:record-decomposition}
 Y = \Gamma_JA+\Gamma_{J^c}B+\Delta,
 \quad\|\Delta\|_{\mathrm F} \le n^{-2},
\end{align}
for all sufficiently large $n$, and
\begin{align} \label{cube:record-budgets}
 \sum_i(\|Ae_i\|_1+\|Be_i\|_1)^2 \le t,
 \quad \#\operatorname{supp}A \le S_1,
 \quad \#\operatorname{supp}B \le S_0.
\end{align}
The error estimate follows from
$$
 \|\Delta\|_{\mathrm F}
 \le \|\Gamma\|_{\mathrm{op}}\sqrt{mn} n^{-4}
 \le \sqrt{C_0n}\sqrt{2n^2} n^{-4} \le n^{-2}
$$
for large $n$.

Fix an absolute constant $C_s$ valid in the first bound of
\eqref{cube:support-estimates}. For numerical parameters $\sigma,r$,
let $\mathcal D_{\sigma,r}$ consist of all pairs $(J,B)$ such that
$J\subseteq\{1,\ldots,m\}$ has $r$ elements, $B$ is an
$(m-r)\times n$ matrix on the $n^{-4}$ grid, and
$$
 \#\operatorname{supp}B \le \lceil C_st^2\sigma L\rceil,
 \quad \sum_i\|Be_i\|_1^2 \le t.
$$
This family is deterministic and finite.

\begin{lemma}[Number of ordinary records] \label{cube:records}
For $\sqrt n \le t \le n$, $\sigma \ge u_0$, and
$r\in\{0,\ldots,m\}$,
\begin{align} \label{cube:record-count}
 \log|\mathcal D_{\sigma,r}| \le Ct^2\sigma L^2.
\end{align}
Every minimizing representation and every threshold set satisfying
\eqref{cube:energy-threshold} yield a member of this family.
\end{lemma}

\begin{proof}
There are at most $2^m$ choices of $J$ and $mn = 2n^2$ coefficient
positions. The column budget bounds each entry by $\sqrt t\le\sqrt n$,
so each occupied position has at most $1+2n^{9/2}$ grid values.
Counting the supports and their entries gives
$$
 \log|\mathcal D_{\sigma,r}|
 \le Cn+C(1+t^2\sigma L)L
 \le Ct^2\sigma L^2.
$$
The last assertion follows from \eqref{cube:record-budgets} and
Proposition~\ref{cube:support}.
\end{proof}

\subsubsection{Conditional Gaussian coordinates}
 \label{cube:conditioning-section}

A full record is a triple $(\sigma,J,B)$, with
$(J,B)\in\mathcal D_{\sigma,r}$, whose numerical parameters satisfy
\eqref{cube:threshold-range} and either \eqref{cube:case-small} or
\eqref{cube:case-large}. A minimum realizes this record if the threshold
construction gives $(\sigma,J)$ and rounding its ordinary coefficients
gives $B$. Realization therefore includes the energy inequalities and
\eqref{cube:record-decomposition}-\eqref{cube:record-budgets}.

Fix a full record. Condition only on $\Gamma_J$. For values of $\Gamma_J$
satisfying \eqref{cube:goodgram}, choose orthogonal coordinates, as a
measurable function of $\Gamma_J$, such that
$$
 \Gamma_J = \begin{pmatrix}R \\0\end{pmatrix},\quad
 \|R\|_{\mathrm{op}} \le 2\sqrt n,\quad
 E = \mathbb R^r\oplus0.
$$
One can use Gram-Schmidt on the columns of $\Gamma_J$ and then on
the coordinate vectors, omitting zero projections. If $r = 0$, take
the identity transformation.

Put $T = (B^TB)^{1/2}$. The bound on $r$ in
\eqref{cube:threshold-range} gives $m-r\ge n$ for large $n$.
Thus the polar factor of $B$ extends to an isometric embedding
$Q_B:\mathbb R^n\to\mathbb R^{m-r}$ with $B = Q_BT$, even when
$B$ is singular. Fix an extension for each record.
In the chosen coordinates,
$$
 Z = \Gamma_{J^c}Q_B = \begin{pmatrix}H \\G\end{pmatrix}
$$
is a standard Gaussian $n\times n$ matrix, conditional on $\Gamma_J$.
The blocks $H\in\mathbb R^{r\times n}$ and
$G\in\mathbb R^{(n-r)\times n}$ are independent. Every realizing
minimum has a representation
\begin{align} \label{cube:conditional-representation}
 Y = \begin{pmatrix}RA+HT \\GT\end{pmatrix}+\Delta
 = ZT+F+\Delta,\quad F = \begin{pmatrix}RA \\0\end{pmatrix}.
\end{align}
The deterministic record satisfies $\operatorname{tr}T^2 \le t$.
For a realizing minimum on $\mathcal G$,
\begin{align} \label{cube:completion-data}
 \operatorname{rank}F \le r,\quad
 \|Y^{-1}\|_{\mathrm F}^2 = \operatorname{tr}M \le t,\quad
 \|\Delta\|_{\mathrm F} \le n^{-2} \le \frac1{4\sqrt t}.
\end{align}
The dependence of these objects is as follows:
$$
\begin{array}{c|l}
\text{fixed data} & \sigma,J,B,T,Q_B\\
\text{determined by }\Gamma_J & R\text{ and the orthogonal coordinates}\\
\text{Gaussian given }\Gamma_J & H,G\text{, independent}\\
\text{allowed to depend on the whole sample} & A,F,\Delta
\end{array}
$$
Theorem~\ref{cube:completion} applies despite the last dependence.
The block form in \eqref{cube:conditional-representation} has a further
use: conditioning later on $G$ leaves $H$ standard Gaussian, allowing
a right test frame to be chosen from $G$ before testing cancellation
in $H$. We never condition on $\mathcal G$ or on the existence of
a minimum realizing the record.

\begin{lemma}[Small exceptional set] \label{cube:small-case}
For each full record satisfying \eqref{cube:case-small}, there is a
Borel event covering all its realizations on $\mathcal G$ whose
conditional probability, given any value of $\Gamma_J$ satisfying its Gram
bounds, is at most $Ce^{-ck^2}$.
\end{lemma}

\begin{proof}
Apply Theorem~\ref{cube:completion} to
\eqref{cube:conditional-representation} with $d = k$.
The record has $\operatorname{tr}T^2 \le t$ and $t \ge \sqrt n$.
If a realizing minimum exists outside the exceptional event,
Corollary~\ref{cube:rank-energy} and \eqref{cube:small-constant} give
$$
 n^2 \le C_{\mathrm{en}}kt^2 \le n^2/2,
$$
a contradiction. Thus the exceptional event covers every realization.
On values of $\Gamma_J$ failing the Gram bounds, define this covering
event to be empty. The coordinate choices above make it Borel, and
its conditional probability has the stated uniform bound.
\end{proof}

\subsection{A Gaussian obstruction for a large exceptional set} \label{cube:large-section}

For a large exceptional set, the inverse spectral estimate bounds
the number of remaining coefficients. We then use the block
decomposition \eqref{cube:conditional-representation}: a right frame
chosen from $G$ and a left frame supplied by the third moment force a
cancellation in $H$. The left frame and the exceptional coefficients
may depend on $H$, so both will be covered by finite nets.

\begin{proposition}[Large exceptional set] \label{cube:large-case}
For each full record satisfying \eqref{cube:case-large}, there is a
Borel event $\mathcal E_{\sigma,J,B}$ containing every sample in
$\mathcal G$ for which some minimizing representation realizes
the record and satisfying
\begin{align} \label{cube:large-probability}
 \mathbb P(\mathcal E_{\sigma,J,B}\mid\Gamma_J) \le Ce^{-cr^2}
\end{align}
for almost every value of $\Gamma_J$ satisfying its Gram bounds.
The constants are independent of the record and of the conditioned
$\Gamma_J$.
\end{proposition}

\paragraph{Reduction in the proof of Proposition~\ref{cube:large-case}.}
Fix a value of $\Gamma_J$ satisfying its Gram bounds.
Unless further conditioning is specified, probabilities in this
subsection are conditional on that value. Apply
Theorem~\ref{cube:completion} to \eqref{cube:conditional-representation}
with $d = r$. Its pointwise condition
\eqref{cube:completion-conclusions} and trace condition
\eqref{cube:aggregate} depend only on $T,t,r$. If either fails,
the exceptional event of the theorem covers every realization and
proves the proposition. Otherwise, set
$$
 T_0 = \max(T,t^{-1/2}I),\quad D_0 = T_0^{-1},
$$
so that
\begin{align} \label{cube:large-envelope}
 \operatorname{tr}D_0^2 \le Crt,\quad
 \|D_0\|_{\mathrm{op}}^2 \le t,\quad
 \operatorname{rank}(TD_0-I) \le \kappa r.
\end{align}
The theorem also gives a Borel event $\mathcal C$ of probability at
most $Ce^{-cr^2}$ outside which every realizing minimum satisfies
\begin{align} \label{cube:large-weighted}
 \sum_{\ell = 1}^n\ell\mu_\ell
 \le Crt\log(en/r)\le CrtL.
\end{align}
This estimate will enter the collective coefficient bound directly.
We add $\mathcal C$ to the final covering event, without conditioning
on its complement.

The next two lemmas construct the test frames. The parameters
$\rho,\theta$ control the left frame, whose dimension is $s$;
the right-frame dimension is $q$, and the small-ball estimate has
$b = sq$ Gaussian coordinates.

\begin{lemma}[Spectral restriction and projection] \label{cube:left-frame}
Let $M\succ0$, $E = \R^r\oplus0$ with $1\le r\le n$, and
$e = \tr(P_EM)$. Suppose $\rho>0$ satisfies
$$
 \tr(P_EM^3)\le\rho^2 e/4.
$$
There are a dyadic $\theta>0$ and a matrix $U$ with $s$ columns such that
\begin{align} \label{cube:left-test}
 \begin{gathered}
 U = \begin{pmatrix}D_L\Sigma \\U_\perp\end{pmatrix},\quad
 D_L^TD_L = I_s,\quad
 \theta I_s\preceq\Sigma^2\preceq2\theta I_s, \\
 U^TM^{-1}U = I_s,\quad
 \|U\|_{\mathrm{op}}^2 \le \rho,
 \end{gathered}
\end{align}
where $\Sigma$ is positive diagonal. Moreover,
\begin{align} \label{cube:left-mass}
 1\le s\le r,\quad \frac e{8r}\le\theta\le\rho,\quad
 e\le C\theta s\log(2+r\rho/e),
\end{align}
with an absolute constant $C$.
\end{lemma}

\begin{proof}
By the hypothesis,
$$
 \operatorname{tr}(P_EM\mathbf1_{(\rho,\infty)}(M))
 \le \rho^{-2}\operatorname{tr}(P_EM^3) \le e/4.
$$
Here $\mathbf1_I(M)$ is the spectral projection of $M$ on $I$.
Thus the restriction $P = \mathbf1_{(0,\rho]}(M)$ retains at least
$3e/4$ of the relevant first moment:
the matrix $P_EM^{1/2}P$ has rank at most $r$ and satisfies
$$
 \|P_EM^{1/2}P\|_{\mathrm F}^2 = \tr(P_EMP)\ge3e/4,\quad
 \|P_EM^{1/2}P\|_{\mathrm{op}}^2\le\rho.
$$
Discarding its squared singular values below $e/(4r)$ loses at most
$e/4$ from their sum. The remaining values lie in $[e/(4r),\rho]$,
which is covered by at most $C\log(2+r\rho/e)$ dyadic intervals
$[\theta,2\theta)$ with $e/(8r)\le\theta\le\rho$. In one interval
the sum of the squared singular values is at least
$ce/\log(2+r\rho/e)$.

Let $Q$ have as columns the corresponding $s$ right singular vectors.
They belong to the range of $P$, since their singular values are positive.
Set $U = M^{1/2}Q$. Then $U^TM^{-1}U = I_s$ and
$\|U\|_{\mathrm{op}}^2\le\rho$. By the singular-value decomposition,
the first $r$ rows of $U$ have the form $D_L\Sigma$ in
\eqref{cube:left-test}. Since the sum of the selected squared singular
values is at most $2s\theta$, we obtain \eqref{cube:left-mass}.
\end{proof}

For the fixed record, set $\rho = C_\rho\sigma^{-1/2}$ with
$C_\rho$ sufficiently large. By \eqref{cube:exceptional-energy},
Lemma~\ref{cube:left-frame} applies to every realizing minimum on
$\mathcal G$. For sufficiently large $n$, $\rho\le1$ and
$e/(8r)\ge\sigma t/(64n)>n^{-3}$. Also $r\rho/e\le8n/(\sigma t)$,
so $\log(2+r\rho/e)\le CL$. Consequently,
$$
 e\le C\theta sL,\quad \theta s\ge\frac{cr\sigma t}{nL}.
$$
Fix a small absolute $c_{*}>0$. We call a pair $(\theta,s)$ admissible
if $s\in\{1,\ldots,r\}$, $\theta$ is dyadic in
$[\frac12 n^{-3},1]$, $\theta \le \rho$, and
$\theta s \ge c_{*}r\sigma t/(nL)$.
There are at most $CnL$ such pairs, in a family fixed by the record.
Choose $c_*$ small enough that every realizing minimum provides one
of them together with \eqref{cube:left-test}-\eqref{cube:left-mass}.
For a fixed pair, only minima providing these additional relations
are considered.

For an admissible pair, define
\begin{align} \label{cube:right-dimension}
 q = \left\lfloor\frac{\alpha\theta n}{\rho}\right\rfloor,
 \quad b = sq.
\end{align}
The admissibility inequalities and \eqref{cube:case-large} give
\begin{align} \label{cube:dimension-scales}
 q\asymp\frac{\alpha\theta n}{\rho},\quad
 q \ge c\frac\alpha\gamma rL,\quad
 b \ge c\frac\alpha\gamma r^2L,
 \quad 2q \le n/2.
\end{align}
Indeed, $s \le r$ implies $\theta \ge c\sigma t/(nL)$, and hence
$$
 \frac{\alpha\theta n}{\rho}
 \ge \frac{c\alpha t\sigma^{3/2}}{L}
 \ge c\frac\alpha\gamma rL.
$$
Thus the integer part changes the expression by at most an absolute
factor. Similarly,
$sq \ge c\alpha r\sigma^{3/2}t/L \ge c(\alpha/\gamma)r^2L$.
Finally, choose $\alpha \le 1/4$. Then $\theta \le \rho$ gives
$2q \le 2\alpha n \le n/2$.

\begin{lemma}[Small singular directions of a Gaussian matrix] \label{cube:right-frame}
Let $T\succeq0$ be a deterministic $n\times n$ matrix, $t>0$, and
$D_0 = \max(T,t^{-1/2}I)^{-1}$. Let $r,q$ be positive integers with
$2\kappa r\le q\le n/4$. Suppose $\tr D_0^2\le Crt$ and
$\rank(TD_0-I)\le\kappa r$, for fixed absolute constants $C,\kappa\ge1$.
If $G$ is an $(n-r)\times n$ standard Gaussian matrix, there are a Borel event $\mathcal R$ and a
Borel choice of an orthonormal $n\times q$ frame $Z_0$, depending only
on $G$ and the deterministic data, such that $V = D_0Z_0$ satisfies,
on $\mathcal R$,
\begin{align} \label{cube:right-test}
 TV = Z_0,\quad
 \|GZ_0\|_{\mathrm{op}} \le Cq/\sqrt n,\quad
 \|V\|_{\mathrm{op}}^2 \le Crt/n,
\end{align}
and
\begin{align} \label{cube:right-failure}
 \mathbb P(\mathcal R^c) \le Ce^{-cqr}.
\end{align}
\end{lemma}

\begin{proof}
Let $\mathcal V$ be the span of the bottom $2q$ right singular
directions of $G$, including its $r$ zero singular directions.
The assumptions give $r<2q<n$, so the positive cutoff has a
spectral gap almost surely. Orthogonal invariance makes
$\mathcal V$ uniform. To estimate the restriction of $G$, append
$r$ independent Gaussian rows. The Gram matrix of the resulting
square matrix dominates $G^TG$, so \eqref{cube:sv-upper} at
index $2q\le n/2$ gives
$$
 \|G|_{\mathcal V}\|_{\mathrm{op}} \le Cq/\sqrt n
$$
outside an event of probability $Ce^{-cq^2}$. This is an event
concerning $G$ alone: the appended rows are used only to bound
its probability. Since $\|D_0\|_{\mathrm{op}}^2\le t$,
Lemma~\ref{cube:haar}, with
$p = 2q$, $u = r$, $v = t$, and the trace hypothesis, gives
$$
 \|D_0|_{\mathcal V}\|_{\mathrm F}^2 \le C_1qrt/n
$$
outside an event of probability at most $e^{-2qr}$, with an absolute
$C_1$. No independence between these two events is needed.

In the compression of $D_0^2$ to $\mathcal V$, remove eigendirections
whose eigenvalues exceed $2C_1rt/n$. There are at most $q/2$ such
directions. Intersect the remaining subspace with
$\ker(TD_0-I)$; this loses at most $\kappa r$ further dimensions.
Since $q\ge2\kappa r$, at least $q$ dimensions remain. Choose an
orthonormal $q$-frame $Z_0$ in this intersection and set $V = D_0Z_0$.
Let $\mathcal R$ be the event on which the two estimates on $\mathcal V$
hold and the necessary ranks and spectral gaps are as above. On this
event the choice gives \eqref{cube:right-test}. Its complement has
probability at most $Ce^{-cq^2}+e^{-2qr}\le Ce^{-cqr}$, proving
\eqref{cube:right-failure}.

For measurability, form the spectral projections, then the projection
onto the required intersection, and apply Gram-Schmidt to projected
coordinate vectors. Kernel projections are Borel, for example by
the formula
$P_{\ker A} = \lim_{\ell\to\infty}(I+\ell A^TA)^{-1}$.
Include the null sets where necessary ranks or gaps fail in
$\mathcal R^c$. Extend $Z_0$ by a fixed orthonormal $q$-frame off
$\mathcal R$ and keep $V = D_0Z_0$ there.
\end{proof}

\begin{proof}[Completion of the proof of Proposition~\ref{cube:large-case}]
For each admissible $(\theta,s)$,
\eqref{cube:dimension-scales} gives $2\kappa r\le q\le n/4$
for large $n$. Lemma~\ref{cube:right-frame} supplies $Z_0,V$ and
an event $\mathcal R_{\theta,s}$ with failure probability at most
$Ce^{-cqr}\le Ce^{-cr^2}$, conditional on $\Gamma_J$.
Both the frame and the event depend only on $G$ and the fixed data.
Consequently, conditional on $\Gamma_J,G$, the matrix $HZ_0$
is standard Gaussian.

We show that every realization on
$\mathcal G\cap\mathcal C^c\cap\mathcal R_{\theta,s}$ forces one
of finitely many small-ball events. Taking their union, then adding
$\mathcal C$ and the failures of $\mathcal R_{\theta,s}$, will give
the required covering event.

\paragraph{The cancellation forced by a realization.}
Suppose a minimum realizes the record on $\mathcal G$, and let $U$
be a left frame from Lemma~\ref{cube:left-frame} for the fixed pair
$(\theta,s)$. On $\mathcal R_{\theta,s}$, multiplying
\eqref{cube:conditional-representation} on the left by $U^T$ and on the right by $V$ gives
$$
 \Sigma D_L^T(HZ_0+RAV)
 = U^T(Y-\Delta)V-U_\perp^TGZ_0.
$$
The normalization $U^TM^{-1}U = I_s$ gives
$\|Y^TU\|_{\mathrm F} = \sqrt s$, and
$\|U_\perp\|_{\mathrm F}\le\sqrt{\rho s}$.
Since $\|\Sigma^{-1}\|_{\mathrm{op}}\le\theta^{-1/2}$,
\eqref{cube:right-test} and $\|\Delta\|_{\mathrm F}\le n^{-2}$ give
\begin{align} \label{cube:cancellation-estimate}
 \|D_L^T(HZ_0+RAV)\|_{\mathrm F}
 \le \frac C{\sqrt\theta}
 \left(\sqrt{\frac{srt}{n}}
 +n^{-2}\sqrt{\frac{\rho rt}{n}}
 +q\sqrt{\frac{\rho s}{n}}\right).
\end{align}
After including the factor $C/\sqrt\theta$, the squared
contributions on the right, divided by $b = sq$, are bounded by
$$
 \frac{Crt}{\theta nq} \le \frac{C\gamma}{\alpha\sigma},
 \quad Cn^{-2},
 \quad\frac{C\rho q}{\theta n} \le C\alpha.
$$
For the first bound, substitute
$q\asymp\alpha\theta n/\rho$,
$\theta \ge c\sigma t/(nL)$,
$r \le C\gamma t\sigma^{3/2}/L^2$, and $\rho = C_\rho/\sqrt\sigma$.
For the middle bound, use
$\rho/\theta \le CnL/(t\sigma^{3/2}) \le Cn$,
$rt/n \le n$, and $b \ge 1$.
Using $(a_1+a_2+a_3)^2\le3(a_1^2+a_2^2+a_3^2)$ in
\eqref{cube:cancellation-estimate}, we obtain
$$
 \|D_L^T(HZ_0+RAV)\|_{\mathrm F}^2
 \le C\left(\frac\gamma{\alpha\sigma}+n^{-2}+\alpha\right)b.
$$
Choose $\alpha>0$ sufficiently small that $C\alpha \le 1/1024$.
This choice depends only on the constants in the preceding estimates.
Since $\sigma \ge u_0 = B_{*}L^3$, the other two terms have total
at most $1/1024$ for sufficiently large $n$. Hence
\begin{align} \label{cube:cancellation}
 \|D_L^T(HZ_0+RAV)\|_{\mathrm F} \le \frac{\sqrt b}{16}.
\end{align}

\paragraph{Finite coefficient and frame nets.}
Assume also that the sample is outside $\mathcal C$.
The exceptional support estimate \eqref{cube:support-estimates},
combined with \eqref{cube:large-weighted}, gives
$$
 S_1 \le n+r^2+C\sqrt{nertL\log(en/r)}
 \le n+r^2+C\sqrt{nert} L.
$$
Since $e \le C\theta sL$, it follows that
$$
 \#\operatorname{supp}A
 \le r^2+n+C\sqrt{nrt\theta s} L^{3/2}.
$$
Round $A$ toward zero on the $n^{-4}$ grid to obtain $A_\#$. Let
$\mathcal A_{\theta,s}$ be the deterministic family of all grid
matrices with this support bound and
$\sum_i\|A_\#e_i\|_1^2 \le t$.
Each entry has magnitude at most $\sqrt t \le \sqrt n$.
Since $r>k$ and $k^2/n\to\infty$, we have $r^2 \ge n$ for large $n$.
Counting supports and grid values as before gives
\begin{align} \label{cube:coefficient-net}
 \log|\mathcal A_{\theta,s}|
 \le C\left(r^2L+\sqrt{nrt\theta s} L^{5/2}\right)
 \le C\left(\frac\gamma\alpha+\frac1{\alpha B_{*}}\right)b.
\end{align}
The two ratios to $b$ satisfy
$$
 \frac{r^2L}{b} \le \frac{C\gamma}\alpha,
 \quad
 \frac{\sqrt{nrt\theta s} L^{5/2}}b
 \le \frac{C\rho L^{5/2}}\alpha
 \sqrt{\frac{rt}{n\theta s}}
 \le \frac{CL^3}{\alpha\sigma} \le \frac C{\alpha B_{*}}.
$$
Rounding preserves the column budget and gives
$\|A-A_\#\|_{\mathrm F}\le\sqrt{rn} n^{-4}$.
Thus $\|R\|_{\mathrm{op}}\le2\sqrt n$ and
\eqref{cube:right-test} imply
$$
 \|R(A-A_\#)V\|_{\mathrm F}
 \le Cr\sqrt{nt} n^{-4} \le Cn^{-2}.
$$

Let $\mathcal N_{r,s}$ be a Frobenius $n^{-4}$-net of the
$r\times s$ orthonormal frames, with centers themselves orthonormal.
A maximal separated family and a volume comparison in
$\mathbb R^{rs}$ give
\begin{align} \label{cube:frame-net}
 \log|\mathcal N_{r,s}| \le CrsL \le \frac{C\gamma b}\alpha,
\end{align}
where the last inequality uses $q/r \ge c\alpha L/\gamma$.
Indeed, every frame has Frobenius norm $\sqrt s$. Balls of
radius $n^{-4}/2$ about separated centers are disjoint and belong
to the ambient ball of radius $\sqrt s+n^{-4}/2$, so their number
is at most $(1+2\sqrt s n^4)^{rs}$.
On $\mathcal G$, the construction of $H$ by an orthogonal change of
coordinates and an isometric embedding gives
$\|H\|_{\mathrm{op}} \le C\sqrt n$.
Since $Z_0$ is an orthonormal $q$-frame, $\|HZ_0\|_{\mathrm F} \le Cn$.
For every allowed $A_\#$, the coefficient budget gives
$\|A_\#\|_{\mathrm F} \le \sqrt t$, so \eqref{cube:right-test} yields
$$
 \|HZ_0+RA_\#V\|_{\mathrm F}
 \le Cn+Ct\sqrt r \le Cn^2.
$$
Choose $D_\#\in\mathcal N_{r,s}$ within Frobenius distance
$n^{-4}$ of $D_L$. Replacing $D_L$ by $D_\#$ changes the
product by at most $Cn^{-2}$ in Frobenius norm. Combining this
with the coefficient rounding and \eqref{cube:cancellation},
every such realization forces, for large $n$,
\begin{align} \label{cube:net-small-ball}
 \|D_\#^THZ_0+D_\#^TRA_\#V\|_{\mathrm F}
 \le \frac{\sqrt b}{4}
\end{align}
for some pair of centers. The event $\mathcal G$ has been used in
this deterministic implication only.

\paragraph{The conditional small-ball bound.}
Condition now on $\Gamma_J,G$. For each fixed pair of centers
$(D_\#,A_\#)$, the first term in \eqref{cube:net-small-ball} is an
$s\times q$ standard Gaussian matrix: $H$ remains standard Gaussian,
and $D_\#,Z_0$ have orthonormal columns. The second term is
deterministic under this conditioning. For a standard Gaussian vector
$\xi\in\mathbb R^b$ and any fixed $a\in\mathbb R^b$,
$$
 \mathbb P\{\|\xi+a\|_2 \le \sqrt b/4\}
 \le e^{b/32}\mathbb E e^{-\|\xi+a\|_2^2/2}
 = e^{b/32}2^{-b/2}e^{-\|a\|_2^2/4} \le e^{-cb}.
$$
Let $c_\mathrm{sb}>0$ denote an absolute constant for which the last
bound is $e^{-c_\mathrm{sb}b}$. By \eqref{cube:coefficient-net} and
\eqref{cube:frame-net}, the total logarithm of the two net sizes is
at most
$$
 C\left(\frac\gamma\alpha+\frac1{\alpha B_{*}}\right)b.
$$
After fixing $\alpha$, choose $\gamma>0$ sufficiently small that
$C\gamma/\alpha \le c_\mathrm{sb}/4$.
Then choose $B_{*}$ sufficiently large that
$C/(\alpha B_{*}) \le c_\mathrm{sb}/4$. A union bound gives
$e^{-c_\mathrm{sb}b/2}$, uniformly over the conditioned data.

Define $\mathcal E_{\sigma,J,B}$ as the union of $\mathcal C$ and,
over all admissible pairs, $\mathcal R_{\theta,s}^c$ and the finite
small-ball events \eqref{cube:net-small-ball} on
$\mathcal R_{\theta,s}$. These are Borel events with the stated
coordinate choices. They cover every realization on $\mathcal G$.
Using \eqref{cube:right-failure}, $b \ge c(\alpha/\gamma)r^2L$, and at most $CnL$
admissible pairs, integration of the conditional estimates yields
$$
 \mathbb P(\mathcal E_{\sigma,J,B}\mid\Gamma_J)
 \le Ce^{-cr^2}
 +\sum_{(\theta,s)}\left(Ce^{-cr^2}+e^{-c'b}\right)
 \le C'e^{-c''r^2}.
$$
Since $r^2\ge n$, the number of admissible pairs is absorbed into
the exponential bound. The exponent can be chosen independently of
$\gamma,B_*$ and $t$: for sufficiently large $n$,
\eqref{cube:dimension-scales} gives $q\ge r$ and $b\ge r^2$,
while the Gaussian and small-ball exponent constants are absolute.
The fixed auxiliary constants affect only the required lower bound
on $n$. Define the covering event to be empty when $\Gamma_J$
fails its Gram bounds; it still covers every realization on
$\mathcal G$.
\end{proof}

\subsection{Proof of the cube lower bound}

We sum the covering probabilities over the deterministic records
and then choose $t$.

\begin{proof}[Proof of Theorem~\ref{cube:main}]
The event $\mathcal G$ fails with probability at most $Ce^{-cn}$.
On $\mathcal G$, any inscribed basis with $\beta(X) \le t/2$ produces
a minimum of value at most $t$ by Theorem~\ref{framework:square}.
Lemma~\ref{cube:threshold} and the rounding construction give a full
record in the small or large case. It remains to sum the covering
events over these finitely many deterministic records.

In the small case, $\sigma = \sigma_0\asymp B_{*}L^3$ and $r \le k$.
Lemma~\ref{cube:records} bounds the logarithm of the total number of
records, including their possible cardinalities, by
$$
 CB_{*}t^2L^5+\log(k+1)
 \le \left(\frac C{\gamma^2B_{*}^2}+o(1)\right)k^2.
$$
Each record has probability at most $Ce^{-ck^2}$ by
Lemma~\ref{cube:small-case}. Choose $B_{*}$ large enough that
$C/(\gamma^2B_{*}^2)$ is smaller than half the absolute exponent constant in
that lemma. For sufficiently large $n$, the record union then has
probability at most $Ce^{-c'k^2}$.

In the large case, fix a numerical pair $(\sigma,r)$.
By \eqref{cube:case-large} and \eqref{cube:record-count},
$$
 \frac{\log|\mathcal D_{\sigma,r}|}{r^2}
 \le \frac{Ct^2\sigma L^2}{r^2}
 \le \frac{CL^6}{\gamma^2\sigma^2} \le \frac C{\gamma^2B_{*}^2}.
$$
Increase $B_{*}$, if necessary, so that the last ratio is smaller
than half the absolute exponent constant in
Proposition~\ref{cube:large-case}. That proposition then bounds the
record union for this pair by $Ce^{-c'r^2} \le Ce^{-c'k^2}$.
There are at most $CnL$ pairs, since $\sigma$ is dyadic and
$r \le C_0n/u_0<n$ for large $n$. Their union still has probability
at most $Ce^{-c''k^2}$ because $k^2/n\to\infty$.
The conditional probability bounds used here are uniform on all
values of $\Gamma_J$ satisfying their Gram bounds, and their covering
events were defined to be empty outside that set. Thus the same
bounds hold after integrating over the values of $\Gamma_J$.

The constant choices have no circular dependence: the cancellation
estimate fixes $\alpha$, the inner nets require $\gamma$ small
relative to $\alpha$, and the net and record counts impose only
lower bounds on $B_*$. Once these are fixed, choose $a_*>0$
sufficiently small and set
\begin{align} \label{cube:final-scale}
 t = a_{*}\frac{n^{2/3}}{L^{5/6}}.
\end{align}
For all sufficiently large $n$, this satisfies
$\sqrt n \le t \le n$, and
$$
 C_{\mathrm{en}}kt^2
 \le C_{\mathrm{en}}\gamma B_{*}^{3/2}t^3L^{5/2}
 = C_{\mathrm{en}}\gamma B_{*}^{3/2}a_{*}^3n^2 \le n^2/2
$$
for sufficiently small $a_{*}$. Thus all energy conditions hold.
The choices are independent of the sample, and the probability bounds
are uniform over admissible $t$.

We have proved that, outside probability
$Ce^{-cn}+Ce^{-c'k^2} \le C'e^{-c''n}$, the event $\mathcal G$ holds
and every inscribed basis has $\beta(X)>t/2$.
Lemma~\ref{cube:reduction} gives
$$
 d_{\mathrm{BM}}(K,B_1^n) \ge c_0t/2
 = \frac{c_0a_{*}n^{2/3}}{2(\log n)^{5/6}}.
$$
The first $n$ Gaussian columns are linearly independent almost surely,
so $K$ and $K^\circ$ are convex bodies. Equation~\eqref{cube:polarity}
proves the probabilistic assertion in sufficiently large dimensions.
Decreasing $c$ if necessary makes $cn^{2/3}/(\log n)^{5/6} \le 1$ in the
remaining finitely many dimensions. Since every Banach-Mazur distance
is at least one, the probabilistic assertion also holds there.
Taking a good realization proves the existence bound.

The argument used finite records and Borel covering events, so no
measurable choice of a minimizing representation is required.
\end{proof}

\appendix

\section{The simultaneous Gaussian event} \label{cube:elementary-appendix}

We prove Lemma~\ref{cube:good}. The parameters are those in
\eqref{cube:parameters}; in particular $u_0 = B_{*}(\log n)^3$.

\begin{proof}
We verify the operator norm, mean absolute projection and small-subset
Gram estimates separately. For a standard Gaussian vector
$g\in\mathbb R^a$,
$$
 \mathbb E e^{u\|g\|_2^2} = (1-2u)^{-a/2}\quad(u<1/2).
$$
Markov's inequality gives exponential upper and lower tails for
$\|g\|_2^2$ at fixed multiples of $a$. We also use a Euclidean
$\delta$-net of $S^{a-1}$ with at most $(1+2/\delta)^a$ points,
obtained by the usual volume comparison for a maximal separated
set. A linear map has norm at most $(1-\delta)^{-1}$ times its
maximum on this net.

For fixed $z\in S^{m-1}$, the vector $\Gamma z$ is standard Gaussian
in $\mathbb R^n$. A $1/2$-net and the upper tail just described prove
\eqref{cube:goodop}, with an exponentially small failure probability,
if $C_0$ is chosen large enough.

For fixed $u\in S^{n-1}$, the vector $\Gamma^Tu$ is standard
Gaussian in $\mathbb R^m$. Its density is bounded by $(2\pi)^{-m/2}$,
and the volume of the $\ell_1$ ball of radius $ma$ is
$2^m(ma)^m/m!$. Thus
$$
 \mathbb P\left\{m^{-1}\|\Gamma^Tu\|_1 \le a\right\}
 \le (Ca)^m.
$$
On \eqref{cube:goodop}, the function
$u\mapsto m^{-1}\|\Gamma^Tu\|_1$ has Lipschitz constant at
most $\|\Gamma\|_{\mathrm{op}}/\sqrt m \le C$. If its minimum on the
unit sphere is below $a/2$, its value is at most $a$ at a point of a
fixed $a/(2C)$-net. Since $m = 2n$, a union bound gives
$$
 (C/a)^n(Ca)^{2n} \le e^{-cn}
$$
for a sufficiently small absolute $a>0$. This proves
\eqref{cube:goodmean} with $c_0 = a/2$.

For the Gram estimates, fix $I$ with $|I| = h\le d_*$.
At each point of a $1/8$-net of $S^{h-1}$, the inequalities
$n/2\le\|\Gamma_Iz\|_2^2\le2n$ fail with probability at most
$Ce^{-cn}$. If they hold at all net points, then
$$
 \|\Gamma_I\|_{\mathrm{op}}
 \le \frac{\sqrt{2n}}{1-1/8}<2\sqrt n,
$$
and
$$
 \min_{\|z\|_2 = 1}\|\Gamma_Iz\|_2
 \ge \sqrt{n/2}-\frac18\frac{\sqrt{2n}}{1-1/8}
 >\frac{\sqrt n}{2}.
$$
There are at most $17^h\binom mh$ net points for each $h$, and
$$
 \log\left(\sum_{h = 1}^{d_{*}}17^h\binom mh\right)
 \le Cd_{*}\log\frac{en}{d_{*}}
 = O\left(\frac{n\log L}{L^3}\right) = o(n).
$$
A union bound proves \eqref{cube:goodgram} simultaneously. Intersecting
the three events proves the lemma.
\end{proof}

\section{The determinant endpoint and the upper bound} \label{cube:upper-appendix}

The determinant endpoint in Section~\ref{framework:sec-endpoints}
recovers the optimization in the upper-bound argument of
Friedland-Naor-Youssef \cite{cube-fny}. We retain the variation and
finite-replacement calculation, quote their covariance completion
lemma, and record the resulting bound. The determinant formulation
below applies to every origin-symmetric convex body.

\begin{lemma}[{\cite[Section~3]{cube-fny}}] \label{cube:upper-radius}
Let $K\subset\R^n$ be an origin-symmetric convex body and let $\tau>0$.
Maximize the determinant of
$$
 A = \sum_{i = 1}^n w_ix_ix_i^T+\tau C
$$
over $x_i\in K$, $w_i\ge0$, $\sum_iw_i = 1$, and
$C\in\conv\{yy^T:y\in K\}$. A maximizer exists and is positive definite.
If $X_K$ has unit ball
$K$, then
$$
 r := \|A^{-1/2}\|_{X_K\to\ell_2^n} \le \sqrt n \tau^{-1/4}.
$$
More precisely, $b = r^2$ satisfies $\tau b^2+n\tau b \le n^2$.
\end{lemma}

\begin{proof}
The feasible set is compact and contains a positive definite matrix.
Splitting a positive-weight summand into identical copies, if necessary,
we may assume every weight is positive. Weight variations
give $x_i^TA^{-1}x_i = a$, independent of $i$. Variations of $C$ toward
$yy^T$ give $y^TA^{-1}y \le \operatorname{tr}(A^{-1}C)$ for $y\in K$.
Since $C\in\mathcal C_K$, it follows that
$$
 b = \max_{y\in K}y^TA^{-1}y = \operatorname{tr}(A^{-1}C),
 \quad n = a+\tau b.
$$
For each column and every $y\in K$, the rank-two determinant formula yields
$$
 \frac{\det(A-w_ix_ix_i^T+w_iyy^T)}{\det A}
 = (1-w_ia)(1+w_i y^TA^{-1}y)+w_i^2(x_i^TA^{-1}y)^2 \le 1.
$$
This identity follows from $\det(I+UV^T) = \det(I+V^TU)$ and does not
require invertibility after removing a column. Subtracting one and
dividing by $w_i$ gives
$$
 (1-w_ia)y^TA^{-1}y+w_i(x_i^TA^{-1}y)^2 \le a.
$$
Summing over the $n$ columns gives $(n-a)y^TA^{-1}y \le na$.
Maximize over $y$ and substitute $n-a = \tau b$ to obtain
$\tau b^2+n\tau b \le n^2$.
\end{proof}

\begin{lemma}[Covariance completion, {\cite[proof of Lemma~2]{cube-fny}}]
 \label{cube:upper-completion}
Let $X$ be an $n$-dimensional normed space, and suppose
$$
 \left(\sum_i\|Se_i\|_X^2\right)^{1/2} \le 1,
 \quad C\in\operatorname{conv}\{yy^T:\|y\|_X \le 1\},
 \quad A = SS^T+\tau C\succ0,
$$
where $S\in\R^{n\times n}$ and $\tau>0$.
Put $r = \|A^{-1/2}\|_{X\to\ell_2^n}$. There is an invertible $M$ with
\begin{align} \label{cube:completion-upper-conclusion}
 MM^T = A,\quad
 \left(\sum_i\|Me_i\|_X^2\right)^{1/2} \le 1+\tau r,
 \quad \|M^{-1}\|_{X\to\ell_2^n} = r.
\end{align}
\end{lemma}

\begin{theorem}[Friedland-Naor-Youssef {\cite{cube-fny}}]
 \label{cube:upper-main}
For every $n$-dimensional real normed space $X$, there is
$M\in\mathrm{GL}(n)$ such that
$$
 \left(\sum_{i = 1}^n\|Me_i\|_X^2\right)^{1/2}
 \|M^{-1}\|_{X\to\ell_2^n}\le2n^{2/3}.
$$
Consequently $d_{\mathrm{BM}}(X,\ell_1^n)\le2n^{2/3}$ and
$d_{\mathrm{BM}}(X,\ell_\infty^n)\le2n^{2/3}$.
\end{theorem}

\begin{proof}
Apply Lemma~\ref{cube:upper-radius} to $K = B_X$.
The matrix $S = [\sqrt{w_i}x_i]_{i = 1}^n$ satisfies
$\sum_i\|Se_i\|_X^2\le1$.
Lemma~\ref{cube:upper-completion} gives an invertible $M$ with
$$
 \left(\sum_i\|Me_i\|_X^2\right)^{1/2}
 \|M^{-1}\|_{X\to\ell_2^n}
 \le r+\tau r^2\le\sqrt n \tau^{-1/4}+n\sqrt \tau.
$$
Choose $\tau = n^{-2/3}$.

For any invertible $M$, put $a_i = \|Me_i\|_X$ and
$U = M\operatorname{diag}(a_i^{-1})$. Then
$\|U\|_{\ell_1^n\to X} \le 1$ and
$$
 \|U^{-1}x\|_1 \le 
 \left(\sum_i a_i^2\right)^{1/2}\|M^{-1}x\|_2.
$$
For the matrix just constructed, this gives
$d_{\mathrm{BM}}(X,\ell_1^n)\le2n^{2/3}$.
Apply this to $X^{*}$ and use polarity for the cube bound.
\end{proof}

\section*{Declaration of AI use}

All mathematical content is due to the author, except for an initial proof of Proposition~\ref{cube:large-case}, which was generated by ChatGPT (GPT-6). The proof presented here has been substantially revised and simplified by the author. The author also wrote the original manuscript. ChatGPT was used to assist with the exposition and organization of the manuscript, identify relevant literature, and compile the bibliography. All mathematical arguments have been personally verified by the author, who takes full responsibility for the originality, correctness, attribution, and integrity of the manuscript.

\end{document}